\documentclass[11pt]{amsart}
\usepackage{etex}

\usepackage{amsmath, amsthm}
\usepackage{eucal}

\usepackage{palatino}
\usepackage{euler}

\usepackage{tabmac}
\usepackage{amssymb}
\usepackage{amscd}
\usepackage{graphicx, xcolor}
\usepackage{amsfonts}
\usepackage{caption}
\usepackage{tikz}
\usepackage{diagxy}
\usepackage{hyperref}
\usepackage{diagrams395}
\usepackage{fullpage}

\usepackage{dsfont}

\usepackage{url}
\usepackage{cite}

\usepackage{dsfont}

\usepackage[margin=1in]{geometry}

\input xy
\xyoption{all}
\UseComputerModernTips

\numberwithin{equation}{section}
\numberwithin{figure}{section}

\newtheorem{theorem}{Theorem}[section]

\newtheorem{lemma}[theorem]{Lemma}
\newtheorem{proposition}[theorem]{Proposition}

\theoremstyle{definition}
\newtheorem{definition}[theorem]{Definition}
\newtheorem{remark}[theorem]{Remark}
\newtheorem{example}[theorem]{Example}

\newcommand{\C}{{\mathbb{C}}}

\newcommand{\Z}{{\mathbb{Z}}}

\newcommand{\h}{\mathfrak{h}}

\newcommand{\clmg}{c_{\lambda,\mu}^{\gamma}}

\begin{document}

\title{Equivariant Quantum Cohomology of the Grassmannian via the Rim Hook Rule}

\author{Anna Bertiger, Elizabeth Mili\'{c}evi\'{c}, and Kaisa Taipale}

\address{Anna Bertiger, Microsoft, Seattle, WA}
\email{anberti@microsoft.com}

\address{Elizabeth Mili\'{c}evi\'{c}, Department of Mathematics \& Statistics, Haverford College, Haverford, PA}
\email{emilicevic@haverford.edu}

\address{Kaisa Taipale, Department of Mathematics, University of Minnesota, Minneapolis, MN}
\email{taipale@math.umn.edu}


\thanks{AB was supported by Cornell University and the University of Waterloo when this work was completed. EM was partially supported by the Institute for Computational and Experimental Research in Mathematics, the Max-Planck-Institut f\"{u}r Mathematik, Simons Collaboration Grant 318716, and NSF Grant 1600982. KT was partially supported by Cornell University and the Mathematical Sciences Research Institute.}







\begin{abstract}
A driving question in (quantum) cohomology of flag varieties is to find non-recursive, positive combinatorial formulas for expressing the product of two classes in a particularly nice basis, called the Schubert basis.  Bertram, Ciocan-Fontanine and Fulton provided a way to compute quantum products of Schubert classes in the Grassmannian of $k$-planes in complex $n$-space by doing classical multiplication and then applying a combinatorial rim hook rule which yields the quantum parameter.  In this paper, we provide a generalization of this rim hook rule to the setting in which there is also an action of the complex torus.  Combining this result with Knutson and Tao's puzzle rule then gives an effective algorithm for computing all equivariant quantum Littlewood-Richardson coefficients.   Interestingly, this rule requires a specialization of torus weights modulo $n$, suggesting a direct connection to the Peterson isomorphism relating quantum and affine Schubert calculus.
\end{abstract}

\maketitle


\section{Introduction}

Quantum cohomology grew out of string theory in the early 1990s. Physicists Candelas, de la Ossa, Green, and Parkes proposed a partial answer to the Clemens conjecture regarding the number of rational curves of given degree on a general quintic threefold, and this brought enormous attention to the mathematical ideas being used by string theorists. A rigorous formulation of (small) quantum cohomology as the deformation of cohomology in which the structure constants count curves satisfying certain incidence conditions was soon developed and extended to a wide class of algebraic varieties and symplectic manifolds; see the survey by Fulton and Pandharipande \cite{FultonPandharipande}. In the mid-1990s Givental proved the conjecture proposed by physicists counting the number of rational curves of given degree on a general quintic threefold \cite{Givental}.   Simultaneously, Givental and Kim introduced equivariant Gromov-Witten invariants and the equivariant quantum cohomology ring \cite{GiventalKim}.

\subsection{Equivariant quantum cohomology of the Grassmannian}

When studying the special case of the Grassmannian of $k$-dimensional subspaces of $\C^n$, the four variants of generalized cohomology discussed here (classical cohomology, quantum cohomology, equivariant cohomology and quantum equivariant cohomology), all have a basis of Schubert classes indexed by partitions with at most $k$ parts each of size at most $n-k$. There are beautiful combinatorial Littlewood-Richardson rules for computing the structure constants in products of this favored basis in classical cohomology.  These \emph{Littlewood-Richardson coefficients} are known to be non-negative integers, and there is an analog of this positivity result in each of the other three contexts.  Since the structure constants for (small) quantum cohomology are enumerative, counting the number of stable maps from rational curves to $Gr(k,n)$ with three marked points mapping into three specified Schubert varieties, they are clearly positive. Graham proved an analog of classical positivity in the equivariant case \cite{Graham}, and Mihalcea proved an analogue for \emph{equivariant quantum Littlewood-Richardson coefficients} \cite{MihalceaAdvances}.  

Prior to the completion of this paper, all known algorithms for computing arbitrary equivariant quantum Littlewood-Richardson coefficients were either recursive or relied on doing computations in a related two-step flag variety. Mihalcea gave the first algorithm for calculating equivariant quantum Littlewood-Richardson coefficients in \cite{MihalceaAdvances} in the form of a (non-positive) recursion. An extension of the puzzle rule of Knutson and Tao \cite{KnutsonTao} to two-step flag varieties has been proved in \cite{BKPT}, and Buch recently generalized this two-step puzzle rule to the equivariant case \cite{Buch}.  The two-step puzzle rule can thus be combined with Buch and Mihalcea's equivariant generalization in \cite{BuchMihalcea} of the ``quantum equals classical'' phenomenon of Buch, Kresch, and Tamvakis \cite{BKT}, in order to compute Schubert structure constants in $QH^*_T(Gr(k,n))$ in a positive, non-recursive manner.  While this paper was near completion, the authors discovered that Gorbounov and Korff had established a \emph{different} non-recursive formula for the equivariant quantum Littlewood-Richardson coefficients, which also does not appeal to two-step flags \cite{GK}. 

In addition to nice Littlewood-Richardson rules, there are analogs of the ring presentation for $H^*(Gr(k,n))$ in terms of Schur polynomials in each of the equivariant and/or quantum contexts. For an overview of the Schur presentation in the classical case, we refer the reader to Fulton's book and the references therein \cite{Fulton}.  The equivariant presentation is also established via the Borel isomorphism, but with the additional perspective of GKM theory as in \cite{KnutsonTao}.  Bertram proved a quantum analogue of the Giambelli and Pieri formulas for the quantum cohomology ring of the Grassmannian \cite{Bertram}, while the equivariant quantum ring presentation was given by Mihalcea\cite{MihalceaTransactions}. In that paper, Mihalcea proves a Giambelli formula which shows that the \emph{factorial Schur polynomials} of \cite{MolevSagan} represent the equivariant quantum Schubert classes.

\subsection{Statement of the main theorem}

  In \cite{BCFF}, Bertram, Ciocan-Fontanine, and Fulton proved a delightful rule for computing the structure constants in the quantum cohomology of $Gr(k,n)$ from the structure constants of the classical cohomology ring of $Gr(k,2n)$.  More specifically, they provide an explicit formula for the quantum Littlewood-Richardson coefficients as signed summations of the classical Littlewood-Richardson coefficients $c_{\lambda, \mu}^{\nu}$ that appear in the expansion of a product of Schubert classes in terms of the Schubert basis.  The algorithm involves removing \emph{rim hooks} from the border strip of the Young diagram for $\nu$ in exchange for picking up signed powers of the quantum variable $q$, and thus became known as the \emph{rim hook rule}.

The main theorem in this paper is an equivariant generalization of the rim hook rule in \cite{BCFF}. In contrast to the pre-existing methods for computing equivariant quantum Littlewood-Richardson coefficients, the method presented in this paper is not recursive and does not rely on related calculations in any two-step flag variety.  In particular, this \emph{equivariant rim hook rule} can be used together with any method of computing equivariant Littlewood-Richardson coefficients, including the Knutson-Tao puzzle package for the computer program Sage \cite{sage}.  The following is an informal statement of Theorem \ref{T:MainTheorem} in the body of this paper.

\begin{theorem} 
The equivariant quantum product of two Schubert classes $\sigma_\lambda \star \sigma_\mu$ in $QH^*_{T}(Gr(k,n))$ can be obtained by computing the equivariant  product of corresponding classes in $H^*_{T}(Gr(k,2n-1))$ and then reducing in a suitable way. This reduction involves both rim hook removal and a specialization of the torus weights modulo $n$.
\end{theorem}

\noindent Here the choice to lift classes to $Gr(k,2n-1)$ is actually quite deliberate, made for reasons of computational convenience in the proof. This is discussed in Remark \ref{why2n-1}.  

To prove this theorem we show that the product defined by this lift and reduction is both associative and coincides with equivariant quantum Chevalley-Monk formula for multiplying by the class corresponding to a single box. Mihalcea's Theorem \ref{T:OneBoxPlusPieri} says that these two conditions suffice to yield a ring isomorphic to $QH^*_T(Gr(k,n))$.  The proof of the equivariant quantum Chevalley-Monk rule is straightforward; the real difficulty lies in proving the associativity statement.  There are two key combinatorial ingredients in the proof of associativity.  The first tool is the abacus model for Young diagrams, which we use to understand the reduction modulo $n$ on the torus weights in Section \ref{S:Abacus}.  In addition, in Section \ref{S:FacSchur} we develop a modification of factorial Schur polynomials, which we call \textit{cyclic factorial Schur polynomials}, in order to relate the classical product in $H^*_{T}(Gr(k,2n-1))$ to the quantum product in $QH^*_{T}(Gr(k,n))$.

\subsection{Directions for future work}  

A stunning result of Peterson proved by Lam and Shimozono \cite{LamShimozono} proves that the equivariant quantum cohomology of any partial flag variety $G/P$ is related to the equivariant homology of the affine Grassmannian. In particular, Peterson's isomorphism says that, up to localization, there is an algebra homomorphism 
\[
H_*^T(Gr_G)_{loc} \twoheadrightarrow QH^*_T(Gr(k,n))_{loc}.
\]
The reduction of torus weights modulo $n$ in the main theorem of this paper also appears in Lam and Shimozono's work \cite{LSDoubleDouble} relating double quantum Schubert polynomials to $k$-double Schur polynomials, which are known to represent equivariant homology classes of the affine Grassmannian \cite{LSkDouble}.  It is the expectation of the authors that cyclic factorial Schur polynomials are the image of the $k$-double Schur polynomials under the Peterson isomorphism.  This connection suggests that the equivariant rim hook rule is a shadow of Peterson's isomorphism and can shed further light on what has become known as the ``quantum equals affine'' phenomenon.  

The authors expect this work to yield equivariant generalizations of several results of Postnikov in \cite{Postnikov} connecting quantum and affine Schubert calculus.  For example, Postnikov provides a quantum Pieri formula, expressed as a sum over \emph{cylindric shapes} using the Jacobi-Trudi formula and an algebraic formulation of the rim hook rule in \cite{BCFF}.  While the authors' equivariant rim hook rule can be combined with any available equivariant Pieri rule for the Grassmannian (e.g. from \cite{Santiago}, \cite{LaksovEqPieri}, \cite{LiRavikumar}, among others) to obtain an equivariant quantum Pieri rule, this approach does not provide a non-negative combinatorial formula.  Instead, the goal would be to find an appropriate combinatorial object which generalizes cylindric shapes to the equivariant context.  In fact, the authors suspect that the cyclic factorial Schur polynomials introduced in this paper are the equivariant analog of the \emph{toric Schur polynomials} in \cite{Postnikov}, which yield quantum Littlewood-Richardson coefficients when expressed in terms of the usual Schur polynomials.

Gorbounov and Korff take an integrable systems approach to studying the equivariant quantum cohomology of the Grassmannian, including an explicit determinantal formula for the equivariant quantum Littlewood-Richardson coefficients; see Corollary 6.29 in \cite{GK}. It is consistent with our rim hook rule and illuminates another aspect of the connection between these non-recursive formulas for equivariant quantum Littlewood-Richardson coefficients and integrable systems.  Moreover, Gorbounov and Korff prove that their vicious and osculating walkers have a concrete representation in the affine nil-Hecke ring, which plays a key role in the proof of Peterson's isomorphism in \cite{LamShimozono}.

\subsection{Organization of the paper}  We begin with a brief review of Schubert calculus on the Grassmannian, with particular emphasis on the equivariant and quantum cohomologies and their polynomial representatives.  In Section \ref{S:RimHook}, we discuss the rim hook rule of \cite{BCFF}, and then provide a precise statement of the equivariant generalization, which is the main result of the paper.  The proof of Theorem \ref{T:MainTheorem} is contained in Section \ref{sectionEQLR}, although we postpone the proofs of three key propositions required for associativity in order not to interrupt the flow of the exposition. Abacus diagrams, which are the first of two important tools for proving associativity, are introduced in Section \ref{S:Abacus}.  Cyclic factorial Schur polynomials are then defined in Section \ref{S:FacSchur}.

\subsection{Acknowledgements} Part of the work for this paper was completed during the semester program on ``Automorphic Forms, Combinatorial Representation Theory, and Multiple Dirichlet Series'' at the Institute for Computational and Experimental Research in Mathematics (ICERM).  The authors are very grateful for the financial support and excellent working conditions at ICERM. The second author wishes to acknowledge the support of the Max-Planck-Institut f\"{u}r Mathematik, and the last author also thanks Cornell University and the Mathematical Sciences Research Institute (MSRI). In addition, several important aspects of the work were tested and implemented using the computer package Sage  \cite{sage}. Thanks to those who wrote the Knutson-Tao puzzles package during Sage Days 45, most notably Franco Saliola, Anne Schilling, Ed Richmond, and Avinash Dalal, in discussions with Allen Knutson.  Finally, the authors are grateful to the anonymous referees, whose comments helped to greatly improve the exposition.


\section{The equivariant rim hook rule} \label{S:RimHook}

\subsection{The Grassmannian and factorial Schur polynomials}\label{sec: notation}

The Grassmannian $Gr(k,n)$ is the complex variety whose points are $k$-planes in $\mathbb{C}^n$. In this paper, we are interested in the equivariant quantum cohomology of the Grassmannian. 

The cohomology of the Grassmannian is governed by the intersection theory of Schubert varieties. A Schubert variety $X_{\lambda}$ is a subvariety of $Gr(k,n)$ satisfying the condition: 
\[
X_{\lambda} := \{ V \in Gr(k,n): \dim (V \cap \mathbb{C}^{n-k-\lambda_i+i}) \geq i, \; \forall \; i\}.
\]
The Schubert varieties of $Gr(k,n)$ are indexed by partitions $\lambda = (\lambda_1, \ldots, \lambda_n)$ with $n-k\geq \lambda_1 \geq \lambda_2 \geq \ldots \geq \lambda_k \geq 0$. We denote the set of such partitions by $\mathcal{P}_{kn}$. We visualize partitions $\lambda$ as \textbf{Young diagrams} of $k$ rows with $\lambda_i$ boxes in the $i^{th}$ row, counting the top row as the first row (this is the English convention), and we shall use this correspondence between partitions and Young diagrams freely.  A \textbf{semi-standard Young tableaux (SSYT) of shape $\lambda$} is a filling of the boxes in the Young diagram with the numbers $1$ to $k$, one number per box,  such that the numbers are weakly increasing in rows proceeding left to right and strictly increasing in columns proceeding top to bottom. Define the \textbf{Schur polynomials} $s_\lambda$ in the variables $x_1, \ldots, x_k$ as
\[
s_\lambda=\sum_T \prod_{\alpha \in T} x_{T(\alpha)},
\] 
where the $T$ are all of the semi-standard fillings of shape $\lambda$ by the numbers 1 through $k$,  the number $T(\alpha)$ is the filling in the box $\alpha$, and the product runs over all boxes in the SSYT.  
Define $e_i$ to be the elementary symmetric polynomials in $x_1, \ldots, x_k$, which can be thought of as $s_{(1)^i}$, and define $h_i$ to be the homogeneous symmetric polynomials $s_{(i)}$.
Then there is an isomorphism 
\begin{align} 
H^*(Gr(k,n)) &\cong \mathbb{Z}[e_1, \ldots, e_k]/ \langle h_{n-k+1}, \ldots, h_n \rangle \label{E:BorelIsom}\\ 
\sigma_{\lambda} &\mapsto s_{\lambda}.
\end{align}

There is a natural $(\mathbb{C}^*)^n$-action on $Gr(k,n)$. Let $T_n \cong (\mathbb{C}^*)^n$ act on $Gr(k,n)$ with weight $t_i$ on the $i^{\text{th}}$ coordinate of $\mathbb{C}^n$. The $T_n$-equivariant cohomology ring $H^*_{T_n}(Gr(k,n))$ is an algebra over the ring $\Lambda := \mathbb{Z}[t_1, \ldots, t_n]$. As a $\Lambda$-module, $H_{T_n}(Gr(k,n))$ again has an additive basis indexed by $\lambda \in \mathcal{P}_{kn}$, which we will also denote by $\{\sigma_{\lambda}\}$. The factorial Schur polynomial $s_{\lambda}(x|t)$ corresponding to a Schubert class $\sigma_{\lambda}$ can then be defined as follows. 
\begin{definition}
The \textbf{factorial Schur polynomial} $s_{\lambda}(x|t)$ corresponding to $\sigma_{\lambda} \in H^*_{T_n}(Gr(k,n))$ is 
\[s_{\lambda}(x|t) = \sum_T \prod_{\alpha \in T}(x_{T(\alpha)}-t_{T(\alpha)+c(\alpha)}) .\] 
The sum is again over all SSYT of shape $\lambda$ filled by the numbers 1 through $k$, and the product is over all boxes $\alpha$ in $T$. Again $T(\alpha)$ gives the filling of $\alpha$ in $T$, and $c(\alpha) = j-i$ when $\alpha$ is the box in the $j^{\text{th}}$ column and $i^{\text{th}}$ row.  
\end{definition} 

We can then define the factorial elementary symmetric functions $e_i(x|t)=s_{(1)^i}(x|t)$ and the factorial homogeneous complete symmetric functions $h_i(x|t)=s_{(i)}(x|t)$.  Notice that if all $t_i=0$, then the polynomial $s_\lambda(x|t)$ specializes to $s_{\lambda}(x)$. Corollary 5.1 and Proposition 5.2 \cite{MihalceaTransactions} then give an isomorphism which generalizes \eqref{E:BorelIsom} to the equivariant setting:
\begin{align}\label{E:EqBorel}
H^*_{T_n}(Gr(k,n)) &\cong \frac{\Lambda [e_1(x|t), \ldots, e_k(x|t)]}{\langle h_{n-k+1}(x|t), \ldots, h_n(x|t) \rangle} \\ 
\sigma_{\lambda} &\mapsto s_{\lambda}(x|t).
\end{align}
In Theorem 1.1 of \cite{MihalceaTransactions}, Mihalcea proves that the $T_n$-equivariant quantum cohomology ring $QH^*_{T_n}(Gr(k,n))$ of the Grassmannian $Gr(k,n)$ is isomorphic to the quotient ring: \begin{equation}\label{eq: qhring}
QH^*_{T_n}(Gr(k,n)) \cong \frac{\Lambda[q, e_1(x|t), \ldots, e_{k}(x|t)]}{\langle h_{n-k+1}(x|t), \ldots, h_{n-1}(x|t), h_{n}(x|t)+(-1)^{k}q \rangle}.
\end{equation} In this case we realize the $s_\lambda(x|t)$ as elements of $QH^*_{T_n}(Gr(k,n)) $ using the \textbf{factorial Jacobi-Trudi formula} due to Chen and Louck; see Theorem 3.3 in \cite{ChenLouck}, or alternatively p. 56 in \cite{Macdonald} for this particular formulation:
\begin{equation}\label{E:FacJacobiTrudi}
 s_{\lambda}(x|t) = \det \left( h_{\lambda_i +j-i}(x|t) \right)_{1 \leq i,j \leq k}.
\end{equation}

\subsection{The rim hook rule}

\begin{figure}[h]
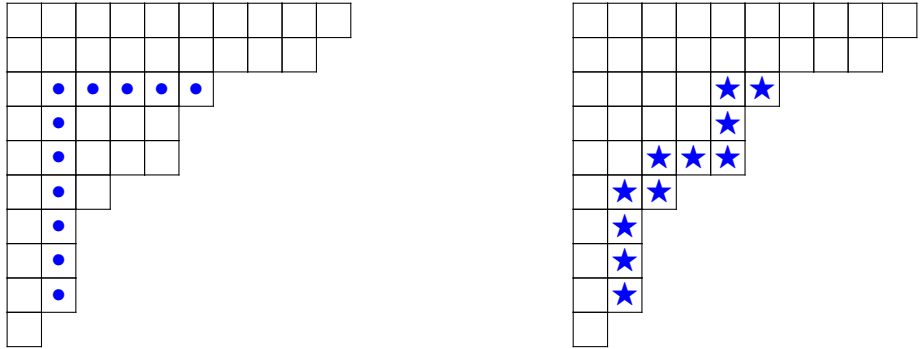

\begin{minipage}[c]{0.45\textwidth}
\[ \tableau[Ys]{ & & & & & & & & & \\ & & & & & & & & \\ & \textcolor{blue}{\bullet} & \textcolor{blue}{\bullet} & \textcolor{blue}{\bullet} &\textcolor{blue}{\bullet} & \textcolor{blue}{\bullet} \\ & \textcolor{blue}{\bullet} & & &   \\  & \textcolor{blue}{\bullet} &   &   &   \\ &\textcolor{blue}{\bullet}  &  \\ &\textcolor{blue}{\bullet}  \\ & \textcolor{blue}{\bullet}  \\ & \textcolor{blue}{\bullet}  \\ \\ } 
\]
\end{minipage}
\begin{minipage}[c]{0.45\textwidth}
\[ \tableau[Ys]{ & & & & & & & & & \\ & & & & & & & & \\ & & & &\textcolor{blue}{\bigstar} & \textcolor{blue}{\bigstar} \\ & & & & \textcolor{blue}{\bigstar}  \\  & &\textcolor{blue}{\bigstar}   & \textcolor{blue}{\bigstar}  & \textcolor{blue}{\bigstar}  \\ &\textcolor{blue}{\bigstar}  &\textcolor{blue}{\bigstar}  \\ &\textcolor{blue}{\bigstar}  \\ & \textcolor{blue}{\bigstar}  \\ & \textcolor{blue}{\bigstar}  \\ \\ } 
\]
\end{minipage}
\caption{An 11-hook (left) and the corresponding 11-rim hook (right).}\label{fig:hooks}
\end{figure}

In \cite{BCFF}, Bertram, Ciocan-Fontanine, and Fulton established a delightful rule presenting quantum Littlewood-Richardson coefficients as signed sums of classical Littlewood-Richardson coefficients. The rim hook algorithm as phrased in \cite{BCFF} does not use the language of lifting Schubert classes, rather carrying out multiplication in the ring of Schur polynomials in the variables $x_1, \ldots, x_k$.  We rephrase the main result from \cite{BCFF} below to draw the most natural parallel to our result.

First we briefly review some required combinatorial terminology. Any box in a Young digram has an associated \textbf{hook}, consisting of all boxes to the right and below the given box, including the box itself.  If the number of boxes in such a hook equals $n$, then we call this an \textbf{$n$-hook}.  Each $n$-hook corresponds to an \textbf{$n$-rim hook}, consisting of the $n$ contiguous boxes running along the border of the Young diagram, starting from the top rightmost box and ending at the bottom-most box of the $n$-hook; see Figure \ref{fig:hooks} for an example.  Removing all possible $n$-rim hooks from a partition $\gamma$ in any order results in the \textbf{$n$-core} for $\gamma$.  We illustrate the process of obtaining the $n$-core of a partition in Figure \ref{fig:core}. Note that there are often multiple ways to remove $n$-rim hooks from a Young diagram; however, the $n$-core is unique.

\begin{figure}[h]
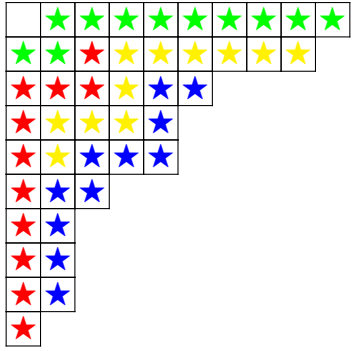

\[ \tableau[Ys]{ & \textcolor{green}{\bigstar}& \textcolor{green}{\bigstar}& \textcolor{green}{\bigstar}&\textcolor{green}{\bigstar} &\textcolor{green}{\bigstar} &\textcolor{green}{\bigstar} & \textcolor{green}{\bigstar}& \textcolor{green}{\bigstar}&\textcolor{green}{\bigstar} \\ \textcolor{green}{\bigstar}&\textcolor{green}{\bigstar} &\textcolor{red}{\bigstar} & \textcolor{yellow}{\bigstar}& \textcolor{yellow}{\bigstar}&\textcolor{yellow}{\bigstar} &\textcolor{yellow}{\bigstar} &\textcolor{yellow}{\bigstar} & \textcolor{yellow}{\bigstar} \\ \textcolor{red}{\bigstar}& \textcolor{red}{\bigstar}& \textcolor{red}{\bigstar}&\textcolor{yellow}{\bigstar} &\textcolor{blue}{\bigstar} & \textcolor{blue}{\bigstar} \\ \textcolor{red}{\bigstar}& \textcolor{yellow}{\bigstar}&\textcolor{yellow}{\bigstar} &\textcolor{yellow}{\bigstar} & \textcolor{blue}{\bigstar}  \\  \textcolor{red}{\bigstar}&\textcolor{yellow}{\bigstar} &\textcolor{blue}{\bigstar}   & \textcolor{blue}{\bigstar}  & \textcolor{blue}{\bigstar}  \\ \textcolor{red}{\bigstar}&\textcolor{blue}{\bigstar}  &\textcolor{blue}{\bigstar}  \\ \textcolor{red}{\bigstar}&\textcolor{blue}{\bigstar}  \\ \textcolor{red}{\bigstar}& \textcolor{blue}{\bigstar}  \\ \textcolor{red}{\bigstar}& \textcolor{blue}{\bigstar}  \\ \textcolor{red}{\bigstar}\\ } 
\]
\caption{The 11-core for (10, 9, 6, 5, 5, 3, 2, 2, 2, 1) is the partition (1).}\label{fig:core}
\end{figure}

\begin{definition}\label{D:ClassicalPhi}
 Define $\varphi: H^*(Gr(k,2n-1)) \to QH^*(Gr(k,n))$ to be the $\Z$-module homomorphism determined by
\begin{align}\label{E:ClassicalPhi}
\sigma_\gamma & \longmapsto   
\begin{cases}
\prod_{i=1}^d\left((-1)^{(\varepsilon_i-k)}q\right) \sigma_{\nu} & \ \  \text{if}\  \nu \in \mathcal{P}_{kn},\\
0 &\ \  \text{if}\  \nu \notin \mathcal{P}_{kn},
 \end{cases} 
\end{align}
for any $\gamma \in \mathcal{P}_{k, 2n-1}$.  Here, we define $\nu$ to be the $n$-core of $\gamma$, noting that $\nu=\gamma$ if $\gamma \in P_{kn}$.  If $\gamma \in \mathcal{P}_{kn}$, the index set for the product is empty, and so the product is 1 in this case. The integer $d$ is the number of $n$-rim hooks removed to get from $\gamma$ to $\nu$, and $\varepsilon_i$ is the height of the $i^{th}$ rim hook removed. 
\end{definition}

Choose the identity map to lift classes in $H^*(Gr(k,n))$ to $H^*(Gr(k,2n-1))$ and denote this lift of $\sigma_{\lambda}$ by $\widehat{\sigma_{\lambda}}$. We also denote by $\sigma_\lambda \star \sigma_\mu$ the quantum product in $QH^*(Gr(k,n))$. The theorem below then follows from \cite{BCFF}.

\begin{theorem}[Main Lemma and Corollary in \cite{BCFF}, rephrased]
Consider $\lambda, \mu \in \mathcal{P}_{kn}$, and write $\widehat{\sigma_{\lambda}} \cdot \widehat{\sigma_{\mu}} = \sum c_{\lambda, \mu}^{\gamma}\sigma_{\gamma} $ in $H^*(Gr(k,2n-1))$. Then, 
\begin{equation}
\sigma_{\lambda} \star \sigma_{\mu} = \sum_{\gamma  \in \mathcal{P}_{k,2n-1}}  c_{\lambda, \mu}^{\gamma} \varphi\left( \sigma_{\gamma} \right) \in QH^*(Gr(k,n)).
\end{equation}
\end{theorem}

\begin{remark}\label{why2n-1} The choice to lift to $H^*(Gr(k,2n-1))$ is quite deliberate, but not obvious from the combinatorial description using Young diagrams. When expressing $QH^*(Gr(k,m))$ using generators and relations analogous to \eqref{E:BorelIsom},
the relations in the quantum ideal include the homogeneous function of degree $m$. If we were to lift to $H^*(Gr(k,2n))$, the reduction map from $H^*(Gr(k,2n))$ to $QH^*(Gr(k,n))$ would require that $h_{2n} \equiv 0 \in H^*(Gr(k,2n))$ map to $0 \neq q^2 \in QH^*(Gr(k,n))$. Choosing $2n-1$ avoids this problem. These same subtleties arise implicitly in Rietsch's description of a subvariety of $GL_n(\C)$ whose coordinate ring is isomorphic to $QH^*(Gr(k,n))$; indeed, Definition 3.2 and Remark 3.3 in \cite{Rietsch} provided the inspiration for many of the ideas in Section \ref{S:FacSchur}.
\end{remark}

\subsection{The equivariant generalization}

Our theorem generalizes the rim hook rule of \cite{BCFF} to the context in which there is also an action of the torus $T_n = (\mathbb{C}^*)^n$.  The equivariant cohomology ring $H^*_{T_n}(Gr(k,n))$ also has a Schubert basis indexed by Young diagrams, and (using the conventions adopted in this paper) the equivariant Littlewood-Richardson coefficients are homogeneous expressions in non-negative sums of polynomials in $ \Z[t_2-t_1, \dots, t_{n}-t_{n-1}]$.  There are combinatorial formulas for explicitly computing the expansions
\begin{equation}
 \sigma_{\lambda} \cdot \sigma_{\mu} = \sum_{\nu} c_{\lambda, \mu}^{\nu}\sigma_{\nu} \in H^*_{T_n}(Gr(k,n)),
\end{equation}  where now each of the cofficients $c_{\lambda, \mu}^{\nu}$ is a homogeneous polynomial in $\Lambda$. One combinatorially pleasant formula is the \textbf{equivariant puzzle rule} of Knutson and Tao \cite{KnutsonTao}, illustrated in Figure \ref{fig: Puzzle}. 
\begin{figure}[h]
\begin{center}
\includegraphics[scale=.45]{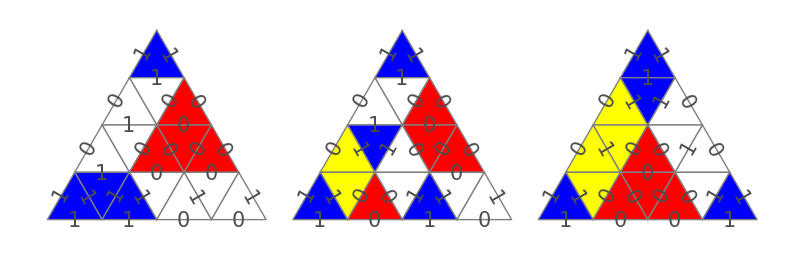}
\caption{Puzzles for computing products in $H^*_T(Gr(2,4))$}\label{fig: Puzzle}
\end{center}
\end{figure} 
The three puzzles shown correspond to the three terms, respectively, in the following product:
\begin{equation*}
\sigma_{\tableau[Yp]{ &  \\}} \cdot \sigma_{\tableau[Yp]{ &  \\}}  =  \sigma_{\tableau[Yp]{ &  \\ & \\ }} + (t_4-t_3)\sigma_{\tableau[Yp]{ &  \\ \\}} +(t_4-t_3)(t_4-t_2)\sigma_{\tableau[Yp]{ &  \\}}.
\end{equation*} 
In particular, the yellow puzzle pieces correspond to the three terms of the form $t_i-t_j$ appearing in this expansion.  (Note that our conventions on the indices for the torus weights are the reverse of those of Knutson and Tao.)

 We again denote the lift of $\sigma_{\lambda}$ from $H^*_{T_n}(Gr(k,n))$ to $H^*_{T_{2n-1}}(Gr(k,2n-1))$ by $\widehat{\sigma_\lambda}$.  We define this lift algebraically using factorial Schur polynomials.  Namely, the lift $\widehat{\sigma_{\lambda}} \in H^*_{T_{2n-1}}(Gr(k,2n-1))$ of $\sigma_{\lambda} \in H^*_{T_n}(Gr(k,n))$ corresponds to the factorial Schur polynomial $s_{\lambda}(x|t)$ in the quotient from \eqref{E:EqBorel} isomorphic to $H^*_{T_{2n-1}}(Gr(k,2n-1))$. By our choice of conventions for the torus weights, if $\lambda \in \mathcal{P}_{kn}$ then none of the weights $t_{n+1}, \ldots, t_{2n-1}$ appear in the expansion given by \eqref{eq: qhring} for the factorial Schur polynomial $s_{\lambda}(x|t)$.  Therefore, on factorial Schur polynomials, our lift is again simply the identity map.

We will now extend the map of $\Z$-modules $\varphi: H^*(Gr(k,2n-1)) \to QH^*(Gr(k,n))$ given in \cite{BCFF} to a map on equivariant cohomology
\begin{equation}\label{E:EqPhiDef}
\varphi :  H^*_{T_{2n-1}}(Gr(k,2n-1)) \longrightarrow QH^*_{T_n}(Gr(k,n)).
\end{equation}  In the equivariant case, the structure constants in $H^*_{T_{2n-1}}(Gr(k,2n-1)$ are polynomials in $t_1, \ldots t_{2n-1}$, but in $QH^*_{T_n}(Gr(k,n))$ the structure constants are polynomials in only $t_1, \ldots t_n$, so we cannot merely use the identity on the structure constants; instead, we use the map $t_i \mapsto t_{i (\text{mod } n)}$.  

\begin{definition}\label{def:EqPhiDef}
For any $t_i$ with $i \in \{1, 2, \dots, 2n-1\}$ and $\gamma \in \mathcal{P}_{k, 2n-1}$, define the map $\varphi$ from Equation \eqref{E:EqPhiDef} to be the $\Z$-module homomorphism determined by 
\begin{align}
t_i & \longmapsto t_{i (\text{mod } n)}  \nonumber \\ 
\sigma_\gamma & \longmapsto   
\begin{cases}
\prod_{i=1}^d\left((-1)^{(\varepsilon_i-k)}q\right) \sigma_{\nu} & \ \  \text{if}\  \nu \in \mathcal{P}_{kn},\\
0 &\ \  \text{if}\  \nu \notin \mathcal{P}_{kn},
 \end{cases} 
\end{align}
 Here, we take the representatives of the congruence classes $\text{mod}\ n$ to be $\{ 1, 2, \dots, n\}$.    
The partition $\nu$ and the statistics $d$ and $\varepsilon_i$ are defined exactly as in Definition \ref{D:ClassicalPhi}.
\end{definition}

\noindent This map $\varphi$ is precisely the same as the non-equivariant version in Definition \ref{D:ClassicalPhi} except that we now also act on the torus weights which appear in the structure constants; this fact justifies our use of the same notation for both maps.

We now state our main result, which we refer to as the Equivariant Rim Hook Rule.

\begin{theorem}[Equivariant Rim Hook Rule]\label{T:MainTheorem}
Let $T_n = (\mathbb{C}^*)^n$ and $T_{2n-1} = (\mathbb{C}^*)^{2n-1}$. Consider any partitions $\lambda, \mu \in \mathcal{P}_{kn}$, and the product expansion $\widehat{\sigma_{\lambda}} \cdot \widehat{\sigma_{\mu}} = \sum c_{\lambda, \mu}^{\gamma}\sigma_{\gamma} $ in $H^*_{T_{2n-1}}(Gr(k,2n-1))$. Then, 
\begin{equation}
\sigma_{\lambda} \star \sigma_{\mu} = \sum_{\gamma \in \mathcal{P}_{k,2n-1}} \varphi \left( c_{\lambda, \mu}^{\gamma}\right) \varphi\left( \sigma_{\gamma} \right) \in QH^*_{T_n}(Gr(k,n)).
\end{equation} 
\end{theorem}

In Section \ref{S:FacSchur}, we also provide an algebraic interpretation of Theorem \ref{T:MainTheorem} and the map $\varphi$ in terms of corresponding rings of factorial Schur polynomials, including providing a direct equivariant analog of the original statements from \cite{BCFF} in Proposition \ref{T: nCores}.

\begin{example}\label{mainex}
We now provide an example which illustrates how to use this theorem to compute quantum equivariant Littlewood-Richardson coefficients.  The computations in equivariant cohomology to provide these examples were done using the Knutson-Tao puzzles package in Sage \cite{sage}.  

For example, to compute $\sigma_{\tableau[Yp]{ &  \\}} \star \sigma_{\tableau[Yp]{ &  \\}} \in QH^*_T(Gr(2,4))$, we first lift the classes to $H^*_T(Gr(2,7))$ via the identity map. We then use Knutson and Tao's equivariant puzzle method to compute this product in $H^*_T(Gr(2,7)):$
\begin{align}
\widehat{\sigma_{\tableau[Yp]{ &  \\}}} \cdot \widehat{\sigma_{\tableau[Yp]{ &  \\}}} & = (t_4-t_3)(t_4-t_2)\sigma_{\tableau[Yp]{ &  \\}} + (t_4-t_3)\sigma_{\tableau[Yp]{ &  \\ \\}} + \sigma_{\tableau[Yp]{ &  \\ & \\ }} \\
\nonumber &\phantom{=} + (t_5+t_4-t_3-t_2)\sigma_{\tableau[Yp]{ & & \\}}+\sigma_{\tableau[Yp]{ &  &\\ \\}} + \sigma_{\tableau[Yp]{ & & &  \\}}.
\end{align}
The map $\varphi$ on torus weights takes $t_i \mapsto t_{i (\text{mod} 4)}$, so that $t_5\mapsto t_1$ while the rest of the torus weights are unchanged.  Now, $\varphi$ acts as the identity on $\sigma_{\tableau[Yp]{ &   \\}}, \sigma_{\tableau[Yp]{ & \\  \\}},$ and $\sigma_{\tableau[Yp]{ & \\ &  \\}}$, since all three of these Young diagrams already fit into a $2 \times 2$ box.  On the other hand, $\sigma_{\tableau[Yp]{ &  &\\}} \mapsto 0$, since this Young diagram neither fits into a $2 \times 2$ box nor contains any removable 4-rim hooks.  Finally, our rim hook rule says that 
\begin{equation}
\sigma_{\tableau[Yp]{ & & \\   \\}}  \mapsto (-1)^{2-2}q = q \quad\quad \text{and} \quad\quad
\sigma_{\tableau[Yp]{ & & & \\}}  \mapsto (-1)^{1-2}q = -q.
\end{equation}
Altogether, Theorem \ref{T:MainTheorem} says that in $QH^*_T(Gr(2,4))$,
\begin{align}
\sigma_{\tableau[Yp]{ &  \\}} \star \sigma_{\tableau[Yp]{ &  \\}} & = (t_4-t_3)(t_4-t_2)\sigma_{\tableau[Yp]{ &  \\}} + (t_4-t_3)\sigma_{\tableau[Yp]{ &  \\ \\}} + \sigma_{\tableau[Yp]{ & \\ &  \\}} + 0 + q - q\\
 & = (t_4-t_3)(t_4-t_2)\sigma_{\tableau[Yp]{ &  \\}} + (t_4-t_3)\sigma_{\tableau[Yp]{ &  \\ \\}} + \sigma_{\tableau[Yp]{ & \\ &  \\}}.
\end{align}
\end{example}


\section{Equivariant Littlewood-Richardson coefficients} \label{sectionEQLR}

We prove Theorem \ref{T:MainTheorem} using the following slight strengthening of Corollary 7.1 in \cite{MihalceaAdvances}.  This result follows directly by combining several statements in \cite{MihalceaAdvances}, but we provide a short proof for the sake of completeness.

\begin{theorem}[Mihalcea \cite{MihalceaAdvances}]\label{T:OneBoxPlusPieri}
Let $\Lambda = \Z[t_1, \dots, t_n]$.  Suppose that $(A, \diamond)$ is a graded, commutative, possibly non-associative $\Lambda[q]$-algebra with unit satisfying the following three properties:
\begin{enumerate}
\item[(a)] $A$ has an additive $\Lambda[q]$-basis $\{ A_{\lambda} \mid \lambda \in \mathcal{P}_{kn} \}$.
\item[(b)] The equivariant quantum Pieri rule holds; \textit{i.e.} 
\begin{equation}\label{E:MihEqQPieri}
A_{\tableau[Yp]{ \\ }} \diamond A_{\lambda} = \sum\limits_{\mu \rightarrow \lambda} A_{\mu} + c_{\lambda, \tableau[Yp]{ \\ }}^{\lambda} A_{\lambda} + q A_{\lambda^-},
\end{equation}
where $\mu \rightarrow \lambda$ denotes a covering relation in $\mathcal{P}_{kn}$, and we define
\begin{equation}\label{E:RecDenom} c_{\lambda, \tableau[Yp]{ \\ }}^{\lambda} = \sum_{i \in U(\lambda)} t_i - \sum_{j=1}^k t_j,
\end{equation}
where $U(\lambda)$ indexes the upward steps in the partition $\lambda$, recorded from southwest to northeast. Here, $A_{\lambda^-}$ equals the basis element corresponding to $\lambda$ with an $(n-1)$-rim hook removed if such a partition exists, and $A_{\lambda^-}$ equals $0$ if such a partition does not exist. 
\item[(c)] Multiplication by one box is associative; \textit{i.e.}
\begin{equation}\label{E:OneBoxAssoc}
(A_{\tableau[Yp]{ \\ }} \diamond A_{\lambda}) \diamond A_{\mu} = A_{\tableau[Yp]{ \\ }} \diamond ( A_{\lambda} \diamond A_{\mu}).
\end{equation}
\end{enumerate}
Then $A$ is canonically isomorphic to $QH^*_T(Gr(k,n))$ as $\Lambda[q]$-algebras.
\end{theorem}

\begin{proof}
If $A$ is a commutative $\Lambda[q]$-algebra with an additive $\Lambda[q]$-basis indexed by $\lambda \in \mathcal{P}_{kn}$ in which the equivariant quantum Pieri rule \eqref{E:MihEqQPieri} and one box associativity \eqref{E:OneBoxAssoc} both hold, then the proof of Proposition 5.1 in \cite{MihalceaAdvances} shows that the equivariant quantum Littlewood-Richardson coefficients satisfy the recursion in Equations (5.2) and (6.1) in \cite{MihalceaAdvances}.  If further $A$ is a graded algebra with unit, then Theorem 2 in \cite{MihalceaAdvances} implies that $A$ is canonically isomorphic to $QH^*_T(Gr(k,n))$.
\end{proof}

In this paper, we denote by $\circ$ our multiplication in $QH^*_{T_n}(Gr(k,n))$ carried out by lifting basis classes, which are the same as the basis classes in $H^*_{T_n}(Gr(k,n))$, to $H^*_{T_{2n-1}}(Gr(k,2n-1))$ and reducing to $QH^*_{T_n}(Gr(k,n))$.  That is, for $\lambda, \sigma \in \mathcal{P}_{kn}$, 
\begin{equation}\label{E:DefCirc}\sigma_{\lambda} \circ \sigma_{\mu} = \varphi(\widehat{\sigma_{\lambda}} \cdot \widehat{\sigma_{\mu}}).
\end{equation} As before, the notation $\cdot$ denotes classical equivariant multiplication in the appropriate ring, and $\star$ will denote the quantum product. An alternative interpretation of Theorem \ref{T:MainTheorem} is that $\circ = \star$.

We will therefore be interested in the algebra $A$ with additive basis $\{\sigma_{\lambda}\}$ indexed by $\lambda \in \mathcal{P}_{kn}$, and the operation defined by this lift-reduction map $\circ$.  To prove that $(A, \circ)$ satisfies the hypotheses of Theorem \ref{T:OneBoxPlusPieri}, we have two primary tasks: to prove the Pieri rule and one box associativity.

\subsection{The equivariant quantum Pieri rule}

We begin by reviewing Mihalcea's quantum equivariant Pieri rule, and we then show that our lift and reduction map agrees with Mihalcea's formula. Denote by $\lambda^-$ the Young diagram obtained by removing an $(n-1)$-rim hook from $\lambda$. Throughout this paper, we shall use the convention that if no such rim hook exists, then $\sigma_{\lambda^-}=0$. When $\lambda_1=n-k$, we also need the related partition $\overline{\lambda} := ((n-k+1), \lambda_2, \dots, \lambda_k)$.  Note that $\overline{\lambda}$ and $\lambda^-$ are related by removal of a single $n$-rim hook.

\begin{theorem}[Theorem 1 \cite{MihalceaAdvances}] \label{T:EqQPieri}
The following Pieri formula holds in $QH_{T_n}^*(Gr(k,n))$:
\begin{equation}\label{E:EqQPieri}
\sigma_{\tableau[Yp]{ \\ }} \star \sigma_\lambda = \sum_{\substack{\mu \to \lambda \\ \mu \in \mathcal{P}_{kn}}} \sigma_\mu + c_{\lambda,  \tableau[Yp]{ \\ }}^\lambda \sigma_\lambda +q\sigma_{\lambda^-}
\end{equation}
\noindent In particular, specializing $q = 0$ in \eqref{E:EqQPieri} recovers the equivariant Pieri rule in $H_{T_n}^*(Gr(k,n))$.
\end{theorem}

\begin{proposition}[Equivariant rim hook Pieri]\label{T:RimHookPieri} For any Young diagram $\lambda \in \mathcal{P}_{kn}$, we have: 
\begin{equation}
\varphi(\widehat{\sigma_{\tableau[Yp]{ \\ }}} \cdot \widehat{\sigma_\lambda})=\sigma_{\tableau[Yp]{ \\ }} \star \sigma_\lambda.
\end{equation}
\end{proposition}

\begin{proof}
If $\lambda_1 \ne n-k$ then the result is immediate: we appeal to the non-quantum equivariant Pieri rule by setting $q=0$ in Theorem \ref{T:EqQPieri} to say that \begin{equation}
 \varphi(\widehat{\sigma_{\tableau[Yp]{ \\ }}} \cdot \widehat{\sigma_\lambda}) =  \sum_{\substack{\mu \to \lambda \\  \mu \in \mathcal{P}_{kn}}} \sigma_\mu + c_{\lambda,  \tableau[Yp]{ \\ }}^\lambda \sigma_\lambda = \sigma_{\tableau[Yp]{ \\ }} \star \sigma_\lambda .
\end{equation}  If $\lambda_1 = n-k$,  recall that $\overline{\lambda}$ is the Young diagram $(n-k+1)\geq \lambda_2\geq  \ldots \geq \lambda_k$ in $\mathcal{P}_{k,2n-1}$.  Then
\begin{eqnarray}
 \varphi (\widehat{\sigma_{\tableau[Yp]{ \\ }}} \cdot \widehat{\sigma_\lambda}) &=& \varphi \left(\sum_{\substack{\mu \to \lambda \\  \mu \in \mathcal{P}_{k,2n-1}}} \sigma_\mu + c_{\lambda, \tableau[Yp]{ \\ }}^\lambda \widehat{\sigma_\lambda} \right)\\
&=& \varphi \left(\sum_{\substack{\mu \to \lambda\\ \mu \in \mathcal{P}_{kn}}} \widehat{\sigma_\mu} + c_{\lambda, \tableau[Yp]{ \\ }}^\lambda \widehat{\sigma_\lambda} +\sigma_{\overline{\lambda}} \right)\\
&=& \sum_{\substack{\mu \to \lambda \\ \mu \in \mathcal{P}_{kn}}} \sigma_\mu + c_{\lambda, \tableau[Yp]{ \\ }}^\lambda \sigma_\lambda +q\sigma_{\lambda^-}\\
&=& \sigma_{\tableau[Yp]{ \\ }} \star \sigma_\lambda.
\end{eqnarray} 
\end{proof}

\subsection{One box associativity}\label{S:OneBoxAssoc}

Recall that we denote by $\circ$ the composition of lifting two Schubert classes from $H^*_{T_n}(Gr(k,n))$ to $H^*_{T_{2n-1}}(Gr(k,2n-1))$, multiplying the classes, and then performing $n$-rim hook reduction as in Equation \eqref{E:DefCirc}. 

\begin{theorem}\label{T:BoxAssoc} 
 For any Young diagrams $\lambda$ and $\mu$ in $\mathcal{P}_{kn}$, we have
\begin{equation}\label{E:BoxAssoc}
\left( \sigma_{\tableau[Yp]{ \\ }} \circ \sigma_\lambda\right) \circ \sigma_\mu = \sigma_{\tableau[Yp]{ \\ }} \circ \left( \sigma_\lambda \circ \sigma_\mu \right).
\end{equation}
\end{theorem}

The proof of Theorem \ref{T:BoxAssoc} requires three fairly serious technical results, which we state now and prove later in order not to interrupt the flow of the exposition.

\begin{proposition}\label{L:eqvtCoeff}
Suppose that $\gamma$ rim hook reduces to $\nu$ by removing $d$ rim hooks each of size $n$. Let $c_{\gamma, \tableau[Yp]{ \\ }}^\gamma$ be a coefficient in $H^*_{T_{2n-1}}(Gr(k,2n-1))$ and $c_{\nu, \tableau[Yp]{ \\ }}^\nu$ be a coefficient in $H^*_{T_n}(Gr(k,n))$.  Then \[\varphi(c_{\gamma, \tableau[Yp]{ \\ }}^\gamma)=c_{\nu, \tableau[Yp]{ \\ }}^\nu=c_{\nu, \tableau[Yp]{ \\ }}^{\nu,0} \in \Lambda.\]
In particular, this implies 
\begin{equation}\label{E:PhiOnU}
\varphi\left( \sum\limits_{i \in U(\gamma)}t_i \right) = \sum\limits_{i \in U(\nu)}t_i.
\end{equation}
\end{proposition}

\begin{proposition}\label{T:deltarimhookoptions}
Suppose that $\gamma\in \mathcal{P}_{k,2n-1}$ reduces to the $n$-core $\nu \in \mathcal{P}_{kn}$ by removing $d$ rim hooks.  Then 
\begin{equation}\label{E:phisum}
\sum\limits_{\substack{\delta \to \gamma \\ \delta \in \mathcal{P}_{k,2n-1}}}\varphi(\sigma_{\delta}) = \sum\limits_{\substack{\epsilon \to \nu \\ \epsilon \in \mathcal{P}_{k,n}}}q^d\sigma_{\epsilon} + q^{d+1}\sigma_{\nu^-} =  q^d\sum\limits_{\substack{\epsilon \to \nu \\ \epsilon \in \mathcal{P}_{k,2n-1}}}\varphi(\sigma_{\epsilon}).
\end{equation}
In particular, note that if $\epsilon \to \nu$ and $\nu \in \mathcal{P}_{k,n}$, then the $n$-core of $\epsilon$ equals $\nu^-$ if and only if $\nu_1 = n-k$ and $\nu_i > 0$ for all $1 \leq i \leq k$ and $\epsilon = \overline{\nu}$; otherwise, $\epsilon$ is an $n$-core.
\end{proposition}

\begin{proposition}\label{T:FacSchur} In $QH^*_{T_n}(Gr(k,n))$, when $\lambda_1 = n-k$, then we have
$\varphi \left(\sigma_{\overline{\lambda}} \cdot \widehat{\sigma_{\mu}} \right) = q \sigma_{\lambda^-} \circ \sigma_{\mu} .$
\end{proposition}

The proof of Propositions \ref{L:eqvtCoeff} and \ref{T:deltarimhookoptions} require the use of abacus diagrams, which we discuss in Section \ref{S:Abacus}.  Proving Proposition \ref{T:FacSchur} inspired the authors to develop a new polynomial model for equivariant quantum cohomology, which we call \textit{cyclic factorial Schur polynomials}.  Cyclic factorial Schur fuctions are discussed in Section \ref{S:FacSchur}.  Assuming these propositions for the moment, we will now proceed with the proof of one box associativity.

\begin{proof}[Proof of Theorem \ref{T:BoxAssoc}]  First we establish some notation.  Suppose that $\lambda, \mu \in \mathcal{P}_{kn}$.  We will write 
\begin{equation}\label{E:lambdamuprod}
\widehat{\sigma_{\lambda}} \cdot \widehat{\sigma_{\mu}} = \sum\limits_{\gamma \in \mathcal{P}_{k,2n-1}}  \clmg \sigma_{\gamma},
\end{equation}
and then define the coefficients $\widetilde{c_{\lambda, \mu}^{\nu, d}}$ via rim hook reduction as 
\begin{equation}\label{E:gammareduction} \varphi \left( \sum\limits_{\gamma \in \mathcal{P}_{k,2n-1}} \clmg \sigma_{\gamma} = \right)  = \sum\limits_{\substack{\nu \in \mathcal{P}_{kn}\\ d \ge 0}}\widetilde{c_{\lambda, \mu}^{\nu, d}}q^d\sigma_{\nu}.
\end{equation} 
We will do a direct calculation, using the two previously stated crucial lemmas, in order to prove the desired equality.  We have
\begin{align}
\nonumber \phantom{a} &   \left(\sigma_{\tableau[Yp]{\\}} \circ \sigma_{\lambda}\right) \circ \sigma_{\mu} 
&
\\
\nonumber
& = \left( \sum\limits_{\substack{\delta \to \lambda \\ \delta \in \mathcal{P}_{kn}}} \sigma_{\delta}  + q\sigma_{\lambda^-} + c_{\lambda, \tableau[Yp]{ \\ }}^\lambda \sigma_{\lambda}\right) \circ \sigma_{\mu}
& \text{Prop. \ref{T:RimHookPieri} \& Thm. \ref{T:EqQPieri}}
\\ 
\nonumber 
& =  \sum\limits_{\substack{\delta \to \lambda \\ \delta \in \mathcal{P}_{kn}}} \sigma_{\delta} \circ \sigma_{\mu} + q \sigma_{\lambda^-} \circ \sigma_{\mu} + c_{\lambda, \tableau[Yp]{ \\ }}^\lambda \sigma_{\lambda} \circ \sigma_{\mu} 
\\
\nonumber
& = \sum\limits_{\substack{\delta \to \lambda \\ \delta \in \mathcal{P}_{kn}}} \sigma_{\delta} \circ \sigma_{\mu} + \varphi \left( \sigma_{\bar{\lambda}} \cdot \sigma_{\mu}\right) + c_{\lambda, \tableau[Yp]{ \\ }}^\lambda \sigma_{\lambda} \circ \sigma_{\mu}. 
& \text{Prop. \ref{T:FacSchur}}.
\end{align}
Note that this middle term equals zero if $\lambda^-$ does not exist, so this last step is nontrivial only when $\lambda_1=n-k$. Now since $\mu \in \mathcal{P}_{kn}$, then we may replace $\sigma_{\mu}$ by $\widehat{\sigma_{\mu}}$ to obtain 
\begin{align}
\nonumber
 \phantom{a} & \sum\limits_{\substack{\delta \to \lambda \\ \delta \in \mathcal{P}_{kn}}} \sigma_{\delta} \circ \sigma_{\mu} + \varphi \left( \sigma_{\bar{\lambda}} \cdot \widehat{\sigma_{\mu}}\right) + c_{\lambda, \tableau[Yp]{ \\ }}^\lambda \sigma_{\lambda} \circ \sigma_{\mu} \\
 &
 \nonumber
= \varphi \left( \sum_{\substack{\delta \to \lambda \\ \delta \in \mathcal{P}_{kn}}} \widehat{\sigma_\delta} \cdot \widehat{\sigma_{\mu}}\right) + \varphi \left( \sigma_{\bar{\lambda}} \cdot \widehat{\sigma_{\mu}}\right)  + \varphi \left( c_{\lambda, \tableau[Yp]{ \\ }}^\lambda \widehat{\sigma_{\lambda}} \cdot \widehat{\sigma_{\mu}} \right) 
\\
\nonumber
\phantom{a} &
 = \varphi \left( \left (\sum_{\substack{\delta \to \lambda \\ \delta \in \mathcal{P}_{k,2n-1}}}\sigma_\delta + c_{\lambda,\tableau[Yp]{ \\ }}^\lambda \widehat{\sigma_\lambda} \right) \cdot \widehat{\sigma_\mu} \right) 
\\
\nonumber
&=\varphi \left( ( \widehat{\sigma_{\tableau[Yp]{ \\ }}} \cdot \widehat{\sigma_\lambda} ) \cdot \widehat{\sigma_\mu} \right).
&
\end{align}

Multiplication in the classical cohomology ring $H^*_T(Gr(k,2n-1))$ is associative, and so here we may write 

\begin{align}
\nonumber
& \phantom{=} \varphi \left( ( \widehat{\sigma_{\tableau[Yp]{ \\ }}} \cdot \widehat{\sigma_\lambda} ) \cdot \widehat{\sigma_\mu} \right)
&
\\
\nonumber
& = \varphi \left( \widehat{\sigma_{\tableau[Yp]{ \\ }}} \cdot ( \widehat{\sigma_\lambda} \cdot \widehat{\sigma_\mu} ) \right)
&  
\\
\nonumber
& = \varphi \left( \widehat{\sigma_{\tableau[Yp]{\\}}} \cdot \sum\limits_{\gamma \in \mathcal{P}_{k,2n-1}}c_{\lambda,\mu}^{\gamma} \sigma_{\gamma} \right) 
& \text{Eq. \eqref{E:lambdamuprod}}
\\
\nonumber
& =  \varphi \left( \sum\limits_{\gamma \in \mathcal{P}_{k,2n-1}} c_{\lambda,\mu}^{\gamma} (\widehat{\sigma_{\tableau[Yp]{\\}}} \cdot \sigma_{\gamma} )\right) 
&
\\
\nonumber
& =  \varphi \left( \sum\limits_{\gamma\in \mathcal{P}_{k,2n-1}}c_{\lambda,\mu}^{\gamma} \left(\sum_{\substack{\delta \to \gamma \\ \delta \in \mathcal{P}_{k,2n-1}}} \sigma_\delta+ c_{\gamma,\; \tableau[Yp]{ \\ }}^\gamma \sigma_{\gamma} \right) \right) 
& \text{Thm. \ref{T:EqQPieri}}
\\
\nonumber
& =  \varphi \left( \sum\limits_{\gamma\in \mathcal{P}_{k,2n-1}}c_{\lambda,\mu}^{\gamma} \sum_{\substack{\delta \to \gamma \\ \delta \in \mathcal{P}_{k,2n-1}}} \sigma_\delta \right) + \varphi \left(  \sum\limits_{\gamma\in \mathcal{P}_{k,2n-1}}c_{\lambda,\mu}^{\gamma} c_{\gamma,\; \tableau[Yp]{ \\ }}^\gamma \sigma_{\gamma} \right)
&
\\
\nonumber
& =  \varphi \left( \sum\limits_{\gamma\in \mathcal{P}_{k,2n-1}}c_{\lambda,\mu}^{\gamma} \sum_{\substack{\delta \to \gamma \\ \delta \in \mathcal{P}_{k,2n-1}}} \sigma_\delta \right)  +  \sum_{\nu \in \mathcal{P}_{kn}, d} \widetilde{c_{\lambda, \mu}^{\nu, d}} q^d c_{\nu,\; \tableau[Yp]{ \\ }}^{\nu}\sigma_{\nu} 
& \text{Prop. \ref{L:eqvtCoeff} \& Eq. \eqref{E:gammareduction}}
\\
\nonumber
&= \sum_{\nu \in \mathcal{P}_{kn}, d} \widetilde{c_{\lambda, \mu}^{\nu, d}} q^d \varphi \left( \sum_{\substack{\epsilon \to \nu \\ \epsilon \in \mathcal{P}_{k,2n-1}}} \sigma_\epsilon   \right) +  \sum_{\nu \in \mathcal{P}_{kn}, d} \widetilde{c_{\lambda, \mu}^{\nu, d}} q^d c_{\nu,\; \tableau[Yp]{ \\ }}^{\nu}\sigma_{\nu}
& \text{Prop. \ref{T:deltarimhookoptions}}
\\
\nonumber
& 
= \sum_{\nu \in \mathcal{P}_{kn}, d} \widetilde{c_{\lambda, \mu}^{\nu, d}} q^d \varphi \left( \sum\limits_{\substack{\epsilon \rightarrow \nu \\ \epsilon \in \mathcal{P}_{k,2n-1} }} \sigma_{\epsilon} + c_{\nu, \tableau[Yp]{\\}}^{\nu} \widehat{\sigma_{\nu}}  \right) 
\\
\nonumber
&= \sum_{\nu \in \mathcal{P}_{kn}, d} \widetilde{c_{\lambda, \mu}^{\nu, d}} q^d \varphi \left(\widehat{\sigma_{\tableau[Yp]{ \\ }}} \cdot \widehat{\sigma_{\nu}} \right) 
& \text{Thm. \ref{T:EqQPieri}}
\end{align}

\begin{align}
\nonumber
&= \sum_{\nu \in \mathcal{P}_{kn}, d} \widetilde{c_{\lambda, \mu}^{\nu, d}} q^d \left(\sigma_{\tableau[Yp]{ \\ }} \circ \sigma_{\nu} \right) 
& \text{}
\\
\nonumber
&=  \sigma_{\tableau[Yp]{ \\ }} \circ \sum_{\nu \in \mathcal{P}_{kn}, d} \widetilde{c_{\lambda, \mu}^{\nu, d}} q^d \sigma_{\nu} 
& \text{}
\\
\nonumber
& = \sigma_{\tableau[Yp]{ \\ }} \circ \varphi \left( \sum\limits_{\gamma \in \mathcal{P}_{k, 2n-1}} c_{\lambda, \mu}^{\gamma} \sigma_{\gamma} \right)
& \text{Eq. \eqref{E:gammareduction}}
\\
\nonumber
& = \sigma_{\tableau[Yp]{ \\ }} \circ (\sigma_\lambda \circ \sigma_\mu).
& \text{Eq. \eqref{E:lambdamuprod}}
\end{align} 

\noindent Altogether we have thus shown that $\left( \sigma_{\tableau[Yp]{ \\ }} \circ \sigma_\lambda\right) \circ \sigma_\mu = \sigma_{\tableau[Yp]{ \\ }} \circ \left( \sigma_\lambda \circ \sigma_\mu \right),$ as desired.
\end{proof}

\begin{proof}[Proof of Theorem \ref{T:MainTheorem}]
Consider the graded, commutative $\Lambda[q]$-algebra $(A, \circ )$ with an additive basis of Schubert classes $\sigma_\lambda$ for all $\lambda \in \mathcal{P}_{kn}$ and product defined by $\sigma_{\lambda} \circ \sigma_{\mu} = \varphi(\widehat{\sigma_{\lambda}} \cdot \widehat{\sigma_{\mu}})$.  By Proposition \ref{T:EqQPieri}, the equivariant quantum Pieri rule holds in $(A, \circ )$.  Theorem \ref{T:BoxAssoc} says that multiplication by the class of a single box is associative.  Therefore, by Theorem \ref{T:OneBoxPlusPieri}, the algebra $(A, \circ )$ is canonically isomorphic to $QH^*_{T_n}(Gr(k,n))$.
\end{proof}


\section{Abacus diagrams and core partitions}\label{S:Abacus}

\subsection{Abacus diagrams, partitions, and $n$-cores}

This section is devoted to the proof of the first two key propositions in the proof of one box associativity, Propositions \ref{L:eqvtCoeff} and \ref{T:deltarimhookoptions}.  The main tool for proving these propositions is the abacus model for a Young diagram; see Section 2.7 in \cite{JamesKerber} for more details.  The abacus model links covers in Young's lattice to the $n$-cores in $\mathcal{P}_{kn}$, the two key players in Propositions \ref{L:eqvtCoeff} and \ref{T:deltarimhookoptions}.

\begin{definition}
 An \textbf{abacus} is an arrangement of the integers into $n$ columns called \textbf{runners}, together with a placement of \textbf{beads} on the integers. The integers are written in order, from left to right and top to bottom, so that each runner is labelled by an equivalence class of the integers modulo $n$. Our convention is to place zero on the left-most runner.  The beads satisfy the condition that there exists an integer $N$ so that there is a bead on every integer before $-N$ and no beads after $N$.  A \textbf{gap} is any non-beaded integer $g$ which precedes some beaded integer $b \geq g$. Note that consecutive non-beaded integers count as multiple gaps, not just one. A bead is said to be \textbf{active} if there exist gaps preceding it.   In our context, only the last $k$ beads are permitted to be active, as this paper only concerns $Gr(k,n)$.  
\end{definition}

Figure \ref{fig:abacus} illustrates two examples of abacus diagrams on $n=3$ runners. The left abacus has gaps at 2 and 4 with active beads on 3 and 5 , while the right abacus has no active beads.

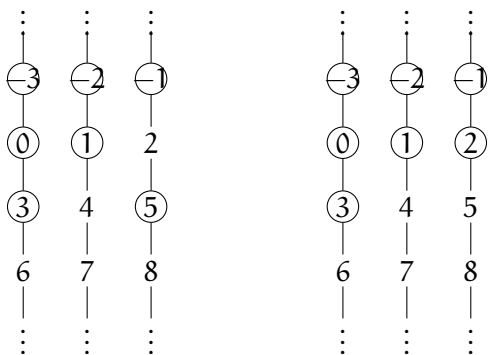
\begin{figure}[h]
\begin{center}
\begin{tikzpicture}[scale=0.85]
\draw (0,-1) node{$\vdots$};
\draw (0,-.75)--(0,-.25);
\draw (0,0) node{$6$};
\draw(0,.25)--(0,.75);
\draw (0,1) circle(7pt) node{$3$};
\draw(0,1.25)--(0,1.75);
\draw (0,2)  circle(7pt) node{$0$};
\draw (0,2.25)--(0,2.75);
\draw (0,3) circle(7pt) node{$-3$};
\draw (0,3.25)--(0,3.75);
\draw (0,4) node{$\vdots$};
\draw (1,-1) node{$\vdots$};
\draw (1,-.75)--(1,-.25);
\draw (1,0) node{$7$};
\draw (1,.25)--(1,.75);
\draw (1,1) node{$4$};
\draw (1,1.25)--(1,1.75);
\draw (1,2) circle(7pt) node{$1$};
\draw(1,2.25)--(1,2.75);
\draw (1,3) circle(7pt) node{$-2$};
\draw (1,3.25)--(1,3.75);
\draw (1,4) node{$\vdots$};
\draw (2,-1) node{$\vdots$};
\draw (2,-.75)--(2,-.25);
\draw (2,0) node{$8$};
\draw (2,.25)--(2,.75);
\draw (2,1) circle(7pt) node{$5$};
\draw (2,1.25)--(2,1.75);
\draw (2,2) node{$2$};
\draw(2,2.25)--(2,2.75);
\draw (2,3) circle(7pt) node{$-1$};
\draw (2,3.25)--(2,3.75);
\draw (2,4) node{$\vdots$};
\draw (5,-1) node{$\vdots$};
\draw (5,-.75)--(5,-.25);
\draw (5,0) node{$6$};
\draw(5,.25)--(5,.75);
\draw (5,1) circle(7pt) node{$3$};
\draw(5,1.25)--(5,1.75);
\draw (5,2)  circle(7pt) node{$0$};
\draw (5,2.25)--(5,2.75);
\draw (5,3) circle(7pt) node{$-3$};
\draw (5,3.25)--(5,3.75);
\draw (5,4) node{$\vdots$};
\draw (6,-1) node{$\vdots$};
\draw (6,-.75)--(6,-.25);
\draw (6,0) node{$7$};
\draw (6,.25)--(6,.75);
\draw (6,1) node{$4$};
\draw (6,1.25)--(6,1.75);
\draw (6,2) circle(7pt) node{$1$};
\draw(6,2.25)--(6,2.75);
\draw (6,3) circle(7pt) node{$-2$};
\draw (6,3.25)--(6,3.75);
\draw (6,4) node{$\vdots$};
\draw (7,-1) node{$\vdots$};
\draw (7,-.75)--(7,-.25);
\draw (7,0) node{$8$};
\draw (7,.25)--(7,.75);
\draw (7,1) node{$5$};
\draw (7,1.25)--(7,1.75);
\draw (7,2) circle(7pt) node{$2$};
\draw(7,2.25)--(7,2.75);
\draw (7,3) circle(7pt) node{$-1$};
\draw (7,3.25)--(7,3.75);
\draw (7,4) node{$\vdots$};
\end{tikzpicture}
\caption{An abacus for $\lambda=(2,1)$ on the left, and an abacus for $\lambda=(0,0)$ on the right.}
\label{fig:abacus}
\end{center}
\end{figure}

To obtain the Young diagram $\lambda$ corresponding to an abacus $\mathcal{A}$, we define its parts  by counting the number of gaps before each of the last $k$ beads.  More precisely, $\lambda_i$ is the number of gaps before the bead which has exactly $i-1$ beads after it, allowing us to construct a partition $\lambda=(\lambda_1, \ldots, \lambda_k)$. For example, counting the number of gaps before each of the two active beads in the left abacus in Figure \ref{fig:abacus} gives $\lambda=(2,1)$.  Note that if $m$ of the last $k$ beads are inactive, as in the right abacus in Figure \ref{fig:abacus}, then the last $m$ parts of $\lambda$ will equal 0. Conversely, one way to create an abacus $\mathcal{A}$ from a partition $\lambda = (\lambda_1, \dots, \lambda_k)$ is to place a bead on every negative integer, and then for $i$ from $1$ to $k$, place a bead on location $\lambda_{k-i+1}+i-1$.

\begin{remark}
There are many abaci which correspond to the same partition $\lambda$. For example, constructing an abacus for the partition $\lambda=(2,1)$ by beading the negative integers and placing beads at $\lambda_{k-i+1}+i-1$ as just described, gives a different abacus than the one shown in Figure \ref{fig:abacus}. Similarly, any abacus in which we translate all beads on the abacus vertically by the same amount produces the same partition. Often \emph{any} abacus diagram corresponding to a fixed partition will suffice for our purposes, but later we do refine this construction in order to make the correspondence one-to-one; see Lemma \ref{T:specialAbacus} below, which constructs a unique preferred abacus diagram for elements of $\mathcal{P}_{kn}$.
\end{remark}

We can now translate much of the requisite combinatorics on partitions into the language of abaci.  For example, the following lemma describes in terms of abacus diagrams the covering relation in Young's lattice given by adding a single box to a partition.

\begin{lemma}\label{T:rightbead}
When $\lambda=(\lambda_1, \ldots, \lambda_k)$ and $\lambda'= (\lambda_1, \ldots, \lambda_{i-1},\lambda_i+1, \lambda_{i+1}, \ldots, \lambda_k)$ are both valid Young diagrams, an abacus for $\lambda'$ can be obtained from an abacus for $\lambda$ by moving a bead on the abacus for $\lambda$ to the next integer.
\end{lemma}

\begin{proof}
The number of boxes in row $i$ of a Young diagram $\lambda$ corresponds to the number of gaps before the $i^{th}$ from the last active bead in an abacus for $\lambda$.  We can add a box in the $i^{th}$ row if and only if there are at least $\lambda_i+1$ boxes in the $(i+1)^{th}$ row of $\lambda$.  This is the case if and only if there is a gap between the $i^{th}$ and $(i+1)^{th}$ active bead in the abacus for $\lambda$.  
\end{proof}

\begin{definition}
An abacus is called \textbf{flush} if each bead has another bead directly above it.    
\end{definition}

For example, the abacus on the left in Figure \ref{fig:abacus} is not flush, but the abacus on the right in Figure \ref{fig:abacus} is flush.  The abacus shown for $\lambda = (0,0)$ is obtained by making the abacus for $\lambda = (2,1)$ in Figure \ref{fig:abacus} flush by moving a bead upwards from 5 to 2. As the following result shows, this corresponds to removing a 3-rim hook from the partition $(2,1)$, resulting in the partition $(0,0)$.

\begin{theorem}[Lemma 2.7.13 and Theorem 2.7.16 \cite{JamesKerber}]\label{T:CoreBead}
$\lambda$ is an $n$-core if and only if every (equivalently any) abacus corresponding to $\lambda$ is flush.  Additionally, removing a single $n$-rim hook from $\lambda$ corresponds to moving one bead up one row on an abacus runner.  
\end{theorem}

We now present two lemmas which use this connection between $n$-cores and abacus diagrams. Lemma \ref{T:beaddist} gives a necessary criterion on abaci for a partition to have $n$-core in $\mathcal{P}_{kn}$, and Lemma \ref{T:specialAbacus} shows that it is possible to choose a unique preferred  abacus diagram for elements of $\mathcal{P}_{kn}$.   We remark that Lemma \ref{T:specialAbacus} is part of a more general phenomenon, and we refer the interested reader to the discussion of balanced, flush abaci in \cite{JamesKerber}.

\begin{lemma}\label{T:beaddist}
If a Young diagram $\gamma$ of at most $k$ parts has $n$-core in $\mathcal{P}_{kn}$ then any abacus for $\gamma$ has each of its last $k$ beads on distinct runners. 
\end{lemma}

\begin{proof}
First suppose that $\gamma$ is itself a Young diagram in $\mathcal{P}_{kn}$, and hence an $n$-core. Assume one of the last $k$ beads occurs on the same runner as another of the last $k$ beads. With at most $k-2$ remaining beads between them, we know that there are $n-1-(k-2) = n-k+1$ gaps between them, giving a partition of width greater than $n-k$. This contradicts the assumption that the partition was in $\mathcal{P}_{kn}$.

If $\gamma \notin \mathcal{P}_{kn}$, we first observe by Theorem \ref{T:CoreBead} that taking an $n$-core does not change the runners on which the active beads of an abacus appear.  Therefore, if $\gamma$ has $n$-core in $\mathcal{P}_{kn}$, the active beads on $\gamma$'s abacus must also be on distinct runners, and therefore so must the last $k$ beads.  
\end{proof}

\begin{lemma}\label{T:specialAbacus}
If $\nu \in \mathcal{P}_{kn}$ then there is a unique abacus for $\nu$ with the last $k$ beads in the row containing $0$ and the last bead on integer $n-1$.  Further, such an abacus must necessarily have inactive beads at every integer $j$ for $j \le -1$.  
\end{lemma}

\begin{proof}
By Lemma \ref{T:beaddist}, the last $k$ beads in any abacus for $\nu$ are on $k$ different runners.  Build an abacus $\mathcal{A}_\nu$ for $\nu$ by starting with an unbeaded set of runners.  Place a bead on integer $n-1$.  Then, for each $i$ from $2$ to $k$ place a bead on location $n-1-(\nu_1-\nu_i)-(i-1)$.  Note that we have now placed $k$ beads on integers between $n-1$ and $n-1-(\nu_1-\nu_k)-(k-1)$.  Since $\nu_1-\nu_k \le n-k$, the $k$ beads we have placed are between $n-1$ and $n-1-(n-k)-(k-1)=0$, all on the row containing $0$.  Finally, place an inactive bead on every integer $j$ for $j \le n-1-\nu_1-k$. 

We have now created an abacus $\mathcal{A}_\nu$ for $\nu$ with last bead at location $n-1$ and all beads in the row containing $0$.  There is also an inactive bead at every location $j \le n-1-\nu_1-k$, and since $\nu_1 \le n-k$, there is an inactive bead at every integer $j \le -1$.  Further, if $\mathcal{A}_\nu'$ is another abacus for $\nu$ with last bead at integer $n-1$, it must be exactly the same abacus as $\mathcal{A}_\nu$.  
\end{proof}

\subsection{Coefficients in equivariant cohomology and $n$-cores}

We are now prepared to prove two of the three key propositions which arise in the course of our proof of one box associativity in Section \ref{S:OneBoxAssoc}.  In the proof of Proposition \ref{L:eqvtCoeff}, we describe the equivariant Littlewood-Richardson coefficient $c_{\gamma, \tableau[Yp]{\\ }}^\gamma$ as sums over active beads on abacus diagrams and hence prove that the equivariant Littlewood-Richardson coefficient $c_{\gamma, \tableau[Yp]{\\ }}^\gamma$ behaves predictably under the map $t_i \mapsto t_{i \operatorname{mod} n}$. 
As a reminder, Proposition \ref{L:eqvtCoeff} says that if $\gamma \in \mathcal{P}_{k,2n-1}$ rim hook reduces to $\nu$ by removing $d$ rim hooks of size $n$, then 
\begin{equation}\label{E:PieriCoeffs}
\varphi(c_{\gamma, \tableau[Yp]{ \\ }}^\gamma)=c_{\nu, \tableau[Yp]{ \\ }}^\nu
\end{equation} and in particular, 
\begin{equation}\label{E:PhiOnU}
\varphi\left( \sum\limits_{i \in U(\gamma)}t_i \right) = \sum\limits_{i \in U(\nu)}t_i.
\end{equation}

\begin{proof}[Proof of Proposition \ref{L:eqvtCoeff}] Recall that $U(\gamma)$ indexes vertical steps for the Young diagram $\gamma$. The $j^{th}$ element of $U(\gamma)$  is $j-1+g(j)$, where $g(j)$ is the number of gaps before the $j^{th}$ active bead in any abacus for $\gamma$. Recall that we count active beads in the same order as the integers they are placed on, and so the first active bead gives the smallest nonzero part in the partition.  Thus, the equivariant Littlewood-Richardson coefficient $c_{\gamma, \tableau[Yp]{ \\ }}^\gamma$ can be written as a sum over locations of active beads on an abacus, independently of the number of runners:
\begin{equation}
c_{\gamma, \tableau[Yp]{ \\ }}^\gamma=\sum_{j=1}^k (t_{j-1+g(j)}-t_j)= \left( \sum_{j=1}^k t_{j-1+g(j)}\right)-(t_1+\cdots t_k).
\end{equation}
 
Notice that the sum $t_1+ \cdots +t_k$ is constant under the map $\varphi$.  By Theorem \ref{T:CoreBead}, removing an $n$-rim hook corresponds to moving a bead up one row on its runner. We therefore see that moving a bead up or down one row on a runner changes the index $j-1+g(j)$ of a torus weight by $n$. Thus, under rim hook reduction indices are constant modulo $n$, and under the map $\varphi$ the two summands of $c_{\gamma, \tableau[Yp]{ \\}}^{\gamma}$ remain constant.  Since $\varphi$ does nothing in the case the abacus is flush, we conclude that $\varphi(c_{\gamma, \tableau[Yp]{ \\ }}^\gamma)=c_{\nu, \tableau[Yp]{ \\ }}^\nu$.
 \end{proof}

Finally, we prove Proposition \ref{T:deltarimhookoptions}. 
Recall that Proposition \ref{T:deltarimhookoptions} states that if $\gamma\in \mathcal{P}_{k,2n-1}$ reduces to the $n$-core $\nu \in \mathcal{P}_{kn}$ by removing $d$ rim hooks each of size $n$, then 
\begin{equation}\label{E:RimHookOpsRestate}
\sum\limits_{\substack{\delta \to \gamma \\ \delta \in \mathcal{P}_{k,2n-1}}}\varphi(\sigma_{\delta}) = \sum\limits_{\substack{\epsilon \to \nu \\ \epsilon \in \mathcal{P}_{k,n}}}q^d\sigma_{\epsilon} + q^{d+1}\sigma_{\nu^-} =  q^d\sum\limits_{\substack{\epsilon \to \nu \\ \epsilon \in \mathcal{P}_{k,2n-1}}}\varphi(\sigma_{\epsilon}).
\end{equation}

\begin{proof}[Proof of Proposition \ref{T:deltarimhookoptions}]

We show the first equality in Equation \eqref{E:RimHookOpsRestate} by showing that each term in the second expression appears in the first expression and that each term appears at most once.  

Let $\mathcal{A}_\gamma$ be an abacus for $\gamma$ on $n$ runners.  We obtain an abacus $\mathcal{A}_\nu$ for $\nu$ by making $\mathcal{A}_\gamma$ flush as in Theorem \ref{T:CoreBead}.  From $\mathcal{A}_\nu$ we can obtain $\mathcal{A}_\epsilon$ for any cover $\epsilon \to \nu$ by moving an active bead on $\mathcal{A}_\nu$ one runner to the right, according to Lemma \ref{T:rightbead}. This will be an $n$-core unless it is the last bead that is moved, making the largest part of the partition $\epsilon$ of length $n-k+1$. In this case, taking the $n$-core of $\epsilon$ we must get $\nu^-$.  In either case, for each $\epsilon$ we can reverse the process of obtaining $\mathcal{A}_\nu$ from $\mathcal{A}_\gamma$ to obtain an abacus for a partition $\delta$ which is a cover of $\gamma$ and has $n$-core $\epsilon$.  

Now we show that each term in the second expression appears at most once in the first expression.    We assume for sake of contradiction that two distinct covering partitions $\delta_1 \to \gamma$ and $\delta_2 \to \gamma$ both rim hook reduce to the same $\epsilon \in \mathcal{P}_{kn}$.  Consider an abacus $\mathcal{A}_\nu$  for $\nu$ as constructed in Lemma \ref{T:specialAbacus}, and fix $\mathcal{A}_\gamma$ an abacus for $\gamma$ that reduces to $\mathcal{A}_\nu$ by making the abacus flush.  By Lemma \ref{T:rightbead}, abaci $\mathcal{A}_{\delta_1}$ and $\mathcal{A}_{\delta_2}$ corresponding to $\delta_1$ and $\delta_2$ can be obtained by moving a bead from $\mathcal{A}_\gamma$ one runner to the right.  

If neither bead moved was on the $n^{\text{th}}$ runner in $\mathcal{A}_\gamma$, then when $\mathcal{A}_{\delta_1}$ and $\mathcal{A}_{\delta_2}$ are made flush to obtain an abacus for $\epsilon$, we must get the unique abacus for $\epsilon$ constructed in Lemma \ref{T:specialAbacus}.  Since $\mathcal{A}_{\delta_1}$ and $\mathcal{A}_{\delta_2}$ differed only in two beads, in order to satisfy Lemma \ref{T:beaddist} we must have $\mathcal{A}_{\delta_1}=\mathcal{A}_{\delta_2}$ and hence $\delta_1=\delta_2$, a contradiction.  

Now suppose that to create $\mathcal{A}_{\delta_1}$ we move the bead on the $n^{\text{th}}$ runner in $\mathcal{A}_\gamma$ to the leftmost runner, while to create $\mathcal{A}_{\delta_2}$ we only moved a bead right one runner in the same row.  If there is no inactive bead on integer $0$ in $\mathcal{A}_\gamma$, then again all of the active beads will be in row $0$ when $\mathcal{A}_{\delta_1}$ and $\mathcal{A}_{\delta_2}$ are made flush. However, these two flush abaci will have unequal $k^{\text{th}}$ parts and thus cannot rim hook reduce to the same $\epsilon$, contradicting our assumption.  If there is an inactive bead at $0$ in $\mathcal{A}_\gamma$, then the bead moved to the first runner to form $\mathcal{A}_{\delta_1}$ will still correspond to the largest part of $\epsilon$ when $\mathcal{A}_{\delta_1}$ is made flush.  In this case the $n$-cores of $\delta_1$ and $\delta_2$ do not have equal first parts, contradicting the assumption that both rim hook reduce to $\epsilon$. 

The second equality in Equation \eqref{E:RimHookOpsRestate} follows from the proof of Proposition \ref{T:RimHookPieri}. 
\end{proof}


\section{Cyclic factorial Schur polynomials}\label{S:FacSchur}

\subsection{Polynomial presentations for equivariant cohomology}  The technical heart of much of this paper is the proof of Proposition \ref{T:FacSchur}, for  which we introduce the concept of \emph{cyclic factorial Schur polynomials}. To complete our proof of the equivariant rim hook rule, we use the fact that the reduction map $\varphi$ on the equivariant cohomology ring with the Schubert basis gives rise to a corresponding map on factorial Schur polynomials.  

Applying the isomorphism from Equation \eqref{E:EqBorel} to $Gr(k, 2n-1)$, we have
\begin{equation}\label{E:ClassicalSchurIsom}
H^*_{T_{2n-1}}(Gr(k,2n-1)) \cong  \frac{\Lambda[e_1(x|t), \ldots, e_k(x|t)] } {\langle h_{2n-k}(x|t) , \ldots, h_{2n-1}(x|t) \rangle.}
\end{equation} 
As discussed in Section \ref{S:RimHook}, the $e_i(x|t)$ are the factorial elementary symmetric polynomials, and the $h_i(x|t)$ are the factorial homogenous complete symmetric polynomials in the variables $x_1, \dots, x_k$ and $t_1, \dots, t_{2n-1}$.  
For brevity, we denote the relevant polynomial ring and ideal by
\begin{align}
\widetilde{R} & =\Lambda[e_1(x|t), \ldots, e_k(x|t)]  \\
\widetilde{J} & = \langle h_{2n-k}(x|t) , \ldots, h_{2n-1}(x|t) \rangle
\end{align} 
so that $H^*_{T_{2n-1}}(Gr(k,2n-1)) \cong  \widetilde{R}/\widetilde{J}$.  Note that $\widetilde{R}$ is generated by the factorial elementary symmetric polynomials in $x_1, \ldots, x_k$ and $t_1, \ldots, t_{2n-1}$. 
Similarly, recall from Equation \eqref{eq: qhring} that
\begin{equation}\label{E:QuantumSchurIsom}
QH^*_{T_n}(Gr(k,n)) \cong \frac{\Lambda[q, e_1(x|t), \ldots, e_k(x|t)]} { \langle h_{n-k+1}(x|t), \ldots, h_n(x|t)+(-1)^kq \rangle},
\end{equation}
where now the $e_i(x|t)$ and $h_i(x|t)$ are the factorial elementary symmetric and homogenous complete symmetric polynomials in the variables $x_1, \dots, x_k$ and $t_1, \dots, t_n$.  We will often write this polynomial ring and its associated ``quantum ideal'' as
\begin{align}
 R & = \Lambda[q, e_1(x|t), \ldots, e_k(x|t)] \\
 J & = \langle h_{n-k+1}(x|t), \ldots, h_n(x|t)+(-1)^kq \rangle
\end{align} 
so that $QH^*_{T_n}(Gr(k,n)) \cong R/J$.  Note that $R$ is generated by the factorial elementary symmetric polynomials in $x_1, \ldots, x_k$ and $t_1, \ldots, t_{n}$.   Under each of the isomorphisms \eqref{E:ClassicalSchurIsom} and \eqref{E:QuantumSchurIsom}, we know by Proposition 5.2 in \cite{MihalceaTransactions} that the class $\sigma_{\lambda}$ corresponds to the factorial Schur polynomial $s_{\lambda}(x|t)$.

\begin{definition}
Given a partition $\lambda \in \mathcal{P}_{k,2n-1}$, we define the \textbf{cyclic factorial Schur polynomial} corresponding to $\lambda$ to be the polynomial in $\widetilde{R}$ obtained from the factorial Schur polynomial  $s_{\lambda}(x|t)$ in the variables $t_1, t_2, \ldots, t_{2n-1}$ by applying the reduction $t_i \mapsto t_{i \operatorname{mod} n}$. We denote the cyclic factorial Schur polynomial corresponding to $\lambda$  by $\overline{s_{\lambda}}(x|t)$ to differentiate it from the original $s_{\lambda}(x|t)$.
\end{definition}

Algebraically, the reduction map from Definition \ref{def:EqPhiDef} \[ \varphi: H^*_{T_{2n-1}}(Gr(k,2n-1)) \longrightarrow QH^*_{T_n}(Gr(k,n))\] then corresponds to a surjective $\Z$-module homomorphism $\widetilde{\varphi}: \widetilde{R}  \longrightarrow  R $
determined by
\begin{align} \nonumber
t_i & \longmapsto t_{i (\text{mod } n)}\\ \label{E:PhiFacSchur}
s_\gamma(x|t) & \longmapsto   
\begin{cases}
\prod_{i=1}^d\left((-1)^{(\varepsilon_i-k)}q\right) s_{\nu}(x|t) & \ \  \text{if}\  \nu \in \mathcal{P}_{kn},\\
0 &\ \  \text{if}\  \nu \notin \mathcal{P}_{kn},
 \end{cases} 
\end{align}
where $\nu$ is the $n$-core of $\gamma$. This map passes to a map on the quotient, which we also denote by $\widetilde{\varphi}: \widetilde{R}/\widetilde{J}  \longrightarrow  R/J $. Notice that $t_i \mapsto t_{i \operatorname{mod} n}$ acts as the identity on any $s_{\lambda}(x|t)$  in $\widetilde{R}$ with $\lambda \in \mathcal{P}_{kn}$, since the torus weight with largest index occurring in a such a factorial Schur polynomial is $k+n-k-1=n-1$ (the filling $T(\alpha)$ of a box $\alpha$ in the Young diagram is bounded by $k$ and the content $c(\alpha)$ is bounded by $n-k-1$).  Thus we can view $s_\lambda(x|t)$ as an element of both $R$ and $\widetilde{R}$ whenever $\lambda \in \mathcal{P}_{kn}$.  To clarify the domain when necessary, we write $\widehat{s}_\lambda(x|t)$ to denote the polynomial $s_\lambda(x|t)$ viewed as an element of the larger ring $\widetilde{R}$.

 \begin{example}  We illustrate the construction of the cyclic factorial Schur polynomials by rewriting Example \ref{mainex} in these terms. In $\widetilde{R}/\widetilde{J}$ with $k=2$ and $n=4$, which is isomorphic to the equivariant cohomology ring of $Gr(2,2\cdot4-1) = Gr(2,7)$, we multiply $s_{\tableau[Yp]{ &  \\}} \cdot s_{\tableau[Yp]{ &  \\}}$ to obtain
\begin{align}
s_{\tableau[Yp]{ &  \\}} \cdot s_{\tableau[Yp]{ &  \\}} & = (t_4-t_3)(t_4-t_2)s_{\tableau[Yp]{ &  \\}} + (t_4-t_3)s_{\tableau[Yp]{ &  \\ \\}} + s_{\tableau[Yp]{ &  \\ & \\ }} \\
\nonumber &\phantom{=} + (t_5+t_4-t_3-t_2)s_{\tableau[Yp]{ & & \\}}+s_{\tableau[Yp]{ &  &\\ \\}} + s_{\tableau[Yp]{ & & &  \\}}.
\end{align}
Then, we reduce the $t_i$ via $t_i \mapsto t_{i \operatorname{mod} 4}$
\[
 (t_4-t_3)(t_4-t_2)\overline{s_{\tableau[Yp]{ &  \\}}} + (t_4-t_3)\overline{s_{\tableau[Yp]{ &  \\ \\}}} + \overline{s_{\tableau[Yp]{ &  \\ & \\ }}}+ (t_1+t_4-t_3-t_2)\overline{s_{\tableau[Yp]{ & & \\}}}+\overline{s_{\tableau[Yp]{ &  &\\ \\}}} + \overline{s_{\tableau[Yp]{ & & &  \\}}},
\]
where $\overline{s_\lambda}$ denotes the cyclic factorial Schur polynomial corresponding the partition $\lambda$ obtained by reducing the torus weights present in $s_\lambda$. In the ring $R/J$, which is isomorphic to the equivariant quantum cohomology ring of $Gr(2,4)$, this expression is equal to 
\[
 (t_4-t_3)(t_4-t_2)s_{\tableau[Yp]{ &  \\ }}+ (t_4-t_3)s_{\tableau[Yp]{ &  \\ \\ }} + s_{\tableau[Yp]{ &  \\ & \\ }}+ (t_1+t_4-t_3-t_2)0+q +(-q)
\]
 This example thus illustrates how cyclic factorial Schur polynomials provide a system of polynomial representatives for equivariant quantum cohomology which do not explicitly contain a quantum parameter.
\end{example}

\subsection{Direct analog of the rim hook rule}   As a reminder, the remaining Proposition \ref{T:FacSchur} states that in $QH^*_{T_n}(Gr(k,n))$, we have 
\begin{equation}\label{E:CyclicFacPropReminder}
\varphi \left(\sigma_{\overline{\lambda}} \cdot \widehat{\sigma_{\mu}} \right) = q \sigma_{\lambda^-} \circ \sigma_{\mu},
\end{equation}  
 where $\mu, \lambda \in \mathcal{P}_{kn}$ with $\lambda_1 = n-k$.  Here we prove a reformulation of this proposition in terms of factorial Schur polynomials. 
 
Recall that Theorem \ref{T:MainTheorem} aims to provide a way to calculate the product of two Schubert classes $\sigma_{\lambda} \star \sigma_{\mu}$ in $QH^*_{T_n}(Gr(k,n))$ using the reduction map $\varphi$.  In particular, 
 \[ \sigma_{\lambda} \star \sigma_{\mu} = \sigma_{\lambda} \circ \sigma_{\mu} = \varphi \left( \widehat{\sigma_{\lambda}} \cdot \widehat{\sigma_{\mu}} \right)\]
 where $\widehat{\sigma_{\lambda}}$ is the lift of $\sigma_{\lambda}$ to $H^*_{T_{2n-1}}(Gr(k,2n-1))$.  This perspective gives rise to the following commutative diagram:
 
 \[\begin{CD}
 H^*_{T_{2n-1}}(Gr(k, 2n-1)) \otimes  H^*_{T_{2n-1}}(Gr(k, 2n-1))@>{\cdot}>>H^*_{T_{2n-1}}(Gr(k,2n-1)) \\
 @A{\hat{\cdot}}AA @V{\varphi}VV\\
 H^*_{T_n}(Gr(k, n)) \otimes  H^*_{T_n}(Gr(k, n))@>{\circ}>>QH^*_{T_n}(Gr(k,n)) 
 \end{CD} \]

\noindent The square above can  be viewed as the front side of the following three-dimensional diagram.  The front face of the cube is connected to the corresponding square in terms of cyclic factorial Schur polynomials on the back side of the diagram by the isomorphisms from \eqref{E:EqBorel}, \eqref{E:ClassicalSchurIsom}, and \eqref{E:QuantumSchurIsom} and the definition of $\widetilde{\varphi}$ in \eqref{E:PhiFacSchur}.

\[ \bfig
\cube|arlb|/@{->}^<>(.6){\cdot}` ->`@{>}_<>(.4){\varphi}`>/%
<1000,1000>[H^*_{T_{2n-1}} \otimes  H^*_{T_{2n-1}}`H^*_{T_{2n-1}}`H^*_{T_n} \otimes H^*_{T_n}`QH^*_{T_n};`\widehat{\cdot}``\circ]%
(400,400)|arlb|/>`@{>}|!{(300,1000);(500,1000)}\hole^<>(.6){\widehat{\cdot}}`>`@{>}%
|!{(1000,500);(1000,300)}\hole_<>(.4){\circ}/<900,900>[\widetilde{R}/\widetilde{J}\otimes\widetilde{R}/\widetilde{J}`\widetilde{R}/\widetilde{J}`R/J \otimes R/J`R/J;\cdot``\widetilde{\varphi}`]%
 |rrrr|/<->`<->`<->`<->/[```]
\efig
\]

\noindent As discussed in Section \ref{S:RimHook}, we choose the unique lift of the factorial Schur polynomial $s_\lambda(x|t)$ from $R$ to $\widetilde{R}$ given by the identity map. Note that in the diagram above we abbreviate $H^*_{T_n}(Gr(k,n))$ by $H^*_{T_n}$ and $QH^*_{T_n}(Gr(k,n))$ by $QH^*_{T_n}$.

In order to prove Proposition \ref{T:FacSchur}, we reformulate Equation \eqref{E:CyclicFacPropReminder} using cylic factorial Schur polynomials.  For such a reformulation to suffice, we need to know that the rightmost square in the above three-dimensional diagram commutes.  To this end, we first prove the following equivariant analog of a key lemma of Bertram, Ciocan-Fontanine, and Fulton. The referenced version of the Main Lemma is phrased in terms of cohomology classes $\sigma_{\lambda}$; we use the commutative cube above to discuss polynomial representatives $s_{\lambda}(x|t)$ instead.

\begin{proposition}\label{T: nCores}[Equivariant generalization of Main Lemma in \cite{BCFF}]
Let $\lambda$ be any partition. In the ring $R/J$, \begin{itemize}
\item[(a)] If $\lambda_{1}>n-k$ and $\lambda$ contains no $n$-rim hook, then $s_{\lambda}(x|t)=0$.
\item[(b)] If $\lambda_{k+1}>0$, then $s_{\lambda}(x|t)=0$.
\item[(c)] If $\lambda$ contains an illegal $n$-rim hook, then $s_{\lambda}(x|t)=0$. (An \textit{illegal} $n$-rim hook is one which starts at the end of a row, moves down and to the left, and ends at an inner corner; namely, when removed it leaves a non-valid partition.)
\item[(d)] If $\nu$ is the result of removing an $n$-rim hook from $\lambda$, then \[s_{\lambda}(x|t)=(-1)^{n-h}qs_{\nu}(x|t),\] where $h$ is the height of the $n$-rim hook removed. 
\end{itemize}
\end{proposition}
 
 From part (d) in particular, we see that the following diagram commutes: 

\begin{diagram}
 H^*_{T_{2n-1}}(Gr(k,2n-1)) & \rTo & \widetilde{R}/\widetilde{J} \\
\dTo{\varphi} &   & \dTo{\widetilde{\varphi}}\\
QH^*_{T_n}(Gr(k,n)) & \rTo & R/J
\end{diagram}

 \noindent Since this fact implies that the entire three-dimensional diagram above is commutative, to conclude our proof of Theorem \ref{T:BoxAssoc}, it  suffices to prove the following reformulation of Proposition \ref{T:FacSchur} using cyclic factorial Schur polynomials.

\begin{proposition}\label{T:AlgProd} For $\mu, \lambda  \in \mathcal{P}_{kn}$ with $\lambda_1 = n-k$, 
\begin{equation}
\widetilde{\varphi}( s_{\overline{\lambda}}(x|t)\widehat{s_\mu}(x|t))= q\widetilde{\varphi}( \widehat{s}_{\lambda^-}(x|t)\widehat{s_\mu}(x|t) ).
\end{equation}
\end{proposition}

\subsection{Jacobi-Trudi formulas and cyclic factorial Schur polynomials}  This section relies heavily on the Jacobi-Trudi formula for expanding factorial Schur polynomials in terms of factorial complete homogeneous polynomials.  Set $\tau^{-s}t$ to be the shifted torus weights $\tau^{-s}t_i=t_{-s+i}$. Recall the factorial Jacobi-Trudi formula from Equation \eqref{E:FacJacobiTrudi}, which states that
\begin{equation*}
s_{\lambda}(x|t) = \det(h_{\lambda_i + j -i}(x|\tau^{1-j}t))_{1 \leq i,j \leq k}.
\end{equation*}

For brevity, let $M_{\nu}$ denote the matrix appearing in the Jacobi-Trudi formula for $s_\nu(x|t)$.  The proofs of Propositions \ref{T: nCores} and \ref{T:AlgProd} first require a lemma on quantum invariance under a shift in torus weights.  We then proceed to an argument using an expansion of the matrices obtained by multiplying the Jacobi-Trudi expansions of factorial Schur polynomials under consideration.

\begin{lemma}[Quantum invariance under shift]\label{lem: quantuminvariance} In the ring $R/J$, 
\begin{equation} h_m(x|\tau^{-s}t) = 0\end{equation}
 for all $s$ and all $m$ such that $n-k < m <n$, and 
 \begin{equation} h_n(x|\tau^{-s}t) = (-1)^{k+1}q  \end{equation}
  for all $s$.
\end{lemma}

\begin{proof}First we establish the base case, invariance under shift of the polynomial $h_{n-k+1}(x|\tau^{-s}t)$. By Equation 1.1 in \cite{MihalceaTransactions}, it is true for all $s$ and given $m$ that
\begin{equation} h_m(x|\tau^{-s}t) = h_m(x|\tau^{-s+1}t) + (t_{m+k-s}- t_{-s+1})h_{m-1}(x|\tau^{-s+1}t).
\end{equation}
Substituting $m = n-k+\ell$ gives the useful equation
\begin{equation}\label{eq: useful}
h_{n-k+\ell}(x|\tau^{-s}t) = h_{n-k+\ell}(x|\tau^{-s+1}t) + (t_{n -k+ \ell+k-s}- t_{-s+1})h_{n-k+\ell-1}(x|\tau^{-s+1}t).
\end{equation}
Notice that for $\ell = 1$, this simplifies to 
\begin{equation}
h_{n-k+1}(x|\tau^{-s}t) = h_{n-k+1}(x|\tau^{-s+1}t) + (t_{n-s+1}- t_{-s+1})h_{n-k}(x|\tau^{-s+1}t).
\end{equation}
Reducing indices on torus weights modulo $n$, the difference $t_{n-s+1}- t_{-s+1} $ is zero, giving  
\begin{equation}
\overline{h}_{n-k+1}(x|\tau^{-s}t) = \overline{h}_{n-k+1}(x|\tau^{-s+1}t)
\end{equation}
 for cyclic factorial Schurs. Because $\overline{h}_{n-k+1}(x|t) =h_{n-k+1}(x|t) \in J$, we see that any shift by $\tau$ is also in $J$. This is the base case for induction. 

For $1<\ell<k$ we use Equation \eqref{eq: useful} and the assertion that $h_{n-k+\ell-1}(x|\tau^{-s+1}t) = 0 \bmod J$ under the inductive hypothesis (invariance of the polynomial under shifting $\tau^{-s}$). Note that $h_{n-k+\ell-1}(x|t) \in J$. This establishes the first statement in the lemma.

For the second statement, $j = n$ implies \[h_n(x|\tau^{-s}t) = h_n(x|\tau^{-s+1}t) + (t_{n+k-s}- t_{-s+1})h_{n-1}(x|\tau^{-s+1}t),\] but by our previous computations $h_{n-1}(x|\tau^{-s+1}t)$ is zero modulo $J$ and so \[h_n(x|\tau^{-s}t) = h_n(x|\tau^{-s+1}t).\]
Therefore, in $R/J$ we have 
 \[h_n(x|\tau^{-s}t) = h_n(x|t) = (-1)^{k+1}q,\]
 for all $s$.
 \end{proof}

We now prove the two key propositions in this section, starting with the equivariant generalization of the rim hook rule from \cite{BCFF}.

\begin{proof}[Proof of Proposition \ref{T: nCores}] Let $\lambda$ be any partition. 

(a) Suppose $\lambda$ does not contain an $n$-rim hook and the width of the partition $\lambda_1>n-k$.  If $\ell$ is the height of the partition (the number of non-zero parts in $\lambda$), we must have $\lambda_1 + \ell <n$, else $\lambda$ would contain an $n$-rim hook. The Jacobi-Trudi formula \[s_{\lambda}(x|t) = \det(h_{\lambda_i + j -i}(x|\tau^{1-j}t))_{1 \leq i,j \leq k}\] reduces to \[s_{\lambda}(x|t) = \det(h_{\lambda_i + j -i}(x|\tau^{1-j}t))_{1 \leq i,j \leq \ell}\] by repeated Laplace expansion along the bottom $k-\ell$ rows of the matrix.  
Since $n-k< \lambda_1 <n-\ell$, the first row of this reduced Jacobi-Trudi matrix consists of the polynomials $h_{\lambda_1 + j -1}(x|\tau^{1-j}t)$ for $1 \leq j \leq \ell$, which are all elements of $J$ and so zero in the ring $R/J$. 

(b) If $\lambda_{k+1}>0$, we have $s_{\lambda}(x|t)=0$ by definition. 

(c) If $\lambda$ contains an illegal $n$-rim hook, a determinantal argument shows $s_{\lambda}(x|t)=0$. If an $n$-rim hook is removed from $\lambda$ starting in row $r$ and ending in row $s$, the resulting shape has row lengths 
\[ 
(\lambda_1, \ldots, \lambda_{r-1}, \lambda_{r+1}-1, \lambda_{r+2}-1, \ldots, \lambda_s-1, \lambda_r-r+s-n, \lambda_{s+1}, \ldots, \lambda_k).
\] 
As discussed in Equation (12) of \cite{BCFF}, the $n$-rim hook is illegal when $\lambda_r -r +s -n = \lambda_{s+1}-1$. Define a factorial Jacobi-Trudi identity 
\[ 
\Upsilon_{m}(x|t) = \det (h_{m_i+j-i}(x|\tau^{1-j}t))_{1 \leq i,j \leq k}
\] for any composition $m=(m_1,\ldots, m_k)$. Apply it to the shape resulting from removing the illegal $n$-rim hook from $\lambda$. Row $s$ in the Jacobi-Trudi matrix is \[ [h_{\lambda_r-r+s-n+1-s}(x|t), \;\; h_{\lambda_r-r+s-n+2-s}(x|\tau^{-1}t), \;\;\ldots, \;\; h_{\lambda_r-r+s-n+k-s}(x|\tau^{1-k}t)].\] Row $s+1$ is \[ [h_{\lambda_{s+1}+1-s-1}(x|t), \;\; h_{\lambda_{s+1}+2-s-1}(x|\tau^{-1}t), \;\;\ldots, \;\; h_{\lambda_{s+1}+k-s-1}(x|\tau^{1-k}t)].\] Since $\lambda_{s+1} = \lambda_r-r+s-n+1$, these are the same, and so the determinant is zero.

(d) If $\lambda$ reduces to $\nu \in \mathcal{P}_{kn}$ by removing $d$ rim hooks of length $n$, we prove that the factorial Schur polynomial $s_{\lambda}(x|t) \in \widetilde{R}$ reduces via $\widetilde{\varphi}$ to $q^d s_{\nu}(x|t)$ in $R/J$. 

We first prove a version of a formula on page 2295 of  \cite{MihalceaTransactions}:
\begin{equation} \label{E:MihalceaStatemt}
\widetilde{\varphi}(s_{\bar{\lambda}}(x|t) )= q s_{\lambda^-}(x|t) 
\end{equation} By the factorial Jacobi-Trudi formula,
 \begin{equation} 
 s_{\bar{\lambda}}(x|t)= \det \left[ \begin{array}{cccc} h_{n-k+1}(x|t) & h_{n-k+2}(x|\tau^{-1}t) & \cdots& h_n(x | \tau^{1-k}t) \\ h_{\lambda_2-1}(x|t) & \ddots & & h_{\lambda_2+k-1}(x|\tau^{1-k}t) \\ \vdots&& \ddots& \vdots \\ h_{\lambda_k-k+1}(x|t)& \cdots  &\cdots & h_{\lambda_k}(x|\tau^{1-k}t) \end{array}\right].\\ 
 \end{equation} 
 Modulo the relations in the ideal $J$, every entry in the top row is zero except for $h_n(x| \tau^{1-k}t)$   $= (-1)^{n-k}q$. 
Expanding the determinant along the top row then gives 
\[ 
\widetilde{\varphi} (s_{\bar{\lambda}}(x|t))= (-1)^{k+1} q \det \left[ \begin{array}{cccc} h_{\lambda_2-1}(x|t) & \cdots & \cdots& h_{\lambda_2+k-2}(x|\tau^{1-k}t) \\ \vdots& \ddots& \ddots& \vdots \\h_{\lambda_k+1-k}(x|t) & \cdots &\cdots & h_{\lambda_k-1}(x|\tau^{1-k}t) \end{array}\right],  
\] 
where the matrix is $k-1 \times k-1$. Recall that $h_n(x|t)+(-1)^kq =0 \operatorname{mod} J$, and the two factors $(-1)^k$ cancel.  As $\lambda^-=(\lambda_2-1, \lambda_3-1, \ldots, \lambda_k-1,0)$, this right-hand side is exactly $q s_{\lambda^-(x|t)}$ by factorial Jacobi-Trudi. This proves that the map $\widetilde{\varphi}$ satisfies Mihalcea's statement \eqref{E:MihalceaStatemt}.

Part (d) goes further and says that if $\lambda$ rim hook reduces to $\nu \in \mathcal{P}_{kn}$ by removing one rim hook of size $n$, then 
\[\widetilde{\varphi}(s_{\lambda}(x|t))=(-1)^{n-h}q s_{\nu}(x|t)\] in $R/J$.  Assume $\lambda_1 \geq n-k+1$. It is easiest to prove this statement by showing that 
\begin{equation}\label{E:LemmaForPartd}
\widetilde{\varphi}( h_{nd+j}(x|t)) = (-1)^{d(n-k-1)}q^d h_{j}(x|t).
\end{equation}
We prove \eqref{E:LemmaForPartd} inductively below in Lemma \ref{T: mini}.

Following the proof of the Main Lemma in \cite{BCFF}, let $m$ be a sequence of integers $m=(m_1, \ldots, m_k)$ and set \begin{equation} \Upsilon_m(x|t) = \det(h_{m_i+j-i}(x|t))_{1\leq i,j\leq k}.\end{equation} By assumption, there is some $\lambda_{\ell}$ for which $\lambda_{\ell} >n-k$; take one of these rows and call it $s$. Apply Lemma \ref{T: mini} to row $s$, as every entry in that row will either have $n-k<\ell-i+j <n$ or $\ell-i+j \geq n$ and so each entry will be zero or will be a multiple of $q$ by Lemma \ref{T: mini}.
 
 Factor out one $q$ from all such rows, and then notice that then we will have the identity \begin{equation} \Upsilon_{m}(x|t) = (-1)^{n-k-1}q \Upsilon_{m''}(x|t),\end{equation} where $m'' = (m_1, \ldots, m_{s-1}, m_s-n, m_{s+1}, \ldots, m_k)$. Rearrange the rows of the matrix to put them "back in order", recalling that if $1 \leq r <s \leq k$ and the $r^{\text{th}}$ row of the matrix is swapped with the $(r+1)^{\text{st}}$, then the $(r+2)^{\text{nd}}$, and so on until the $s^{\text{th}}$ row, we have \begin{equation}\Upsilon_m(x|t) = (-1)^{s-r}\Upsilon_{m'}(x|t),\end{equation} where $m' = (m_1, \ldots, m_{r-1}, m_{r+1}-1, \ldots, m_s-1, m_{r}-r+s, m_{s+1}, \ldots, m_k)$. Combining these two processes of factoring out $q$ and rearranging rows so that the resulting $m'''$ is strictly decreasing, we get \begin{equation}\Upsilon_m(x|t) = (-1)^{n-k-1+s-r}q\Upsilon_{m'''}(x|t), \end{equation} with $m''' = (m_1, \ldots, m_{r-1},m_{r+1}-1, \ldots, m_s-1, m_r-r+s-n, m_{s+1}, \ldots, m_k)$ -- $\lambda$ with a rim hook removed. (If at any point there was no rim hook to remove, the process would halt, and an illegal rim hook would result in determinant zero.) If another rim hook can be removed, repeat. This gives part (d).
\end{proof}

\begin{lemma}\label{T: mini} In the ring $R/J$, 
\[\widetilde{\varphi}(h_{nd+j}(x|t)) = (-1)^{d(n-k-1)}q^d h_{j}(x|t) .\]
\end{lemma}

\begin{proof} 
First we present relevant identities and a short example, and then we induct on the degree of the homogeneous polynomial using degree $n$ as the base case. We know that $h_n(x|t) = (-1)^{k}q$ in $R/J$ as this is a relation in $J$.

Use formula (2.10) in \cite{MihalceaTransactions} to formally write 
\begin{equation}\sum_{r=0}^k (-1)^r e_r(x|t) h_{s-r}(x| \tau^{1-s}t) = 0
\end{equation}
for each positive $s$. An alternative and equivalent formulation is \begin{equation}\label{eq: sumzero}  \sum_{r=0}^k (-1)^r e_r(x|\tau^{s-1}t) h_{s-r}(x|t) = 0.\end{equation} As an example of how the proof will go, look at the case $s=n+1$. Then,  
\begin{align}
\phantom{a} & \sum_{r=0}^k (-1)^r e_r(x|t)h_{n+1-r}(x|\tau^{-n}t)  \\
 \nonumber
& = \overline{h}_{n+1}(x|t) -\overline{h}_{n}(x|t)e_1(x|t) + \ldots \pm \overline{h}_{n-k+1}(x|t)e_k(x|t) = 0.\end{align} 
The relations in the ideal $J$ eliminate most terms in the middle expression and give\begin{equation} \overline{h}_{n+1}(x|t) = \overline{h}_n(x|t) e_1(x|t).\end{equation} This can be rewritten using the quantum relation to give \begin{equation} \widetilde{\varphi}(h_{n+1}(x|t))= (-1)^k q e_1(x|t). \end{equation}

We now induct.  Let $m>n$, and write $m=nd+j$ for positive $d$ and $0\leq j \leq n-1$. Use equation \eqref{eq: sumzero} again to write  \begin{equation}\sum_{r=0}^k (-1)^r e_r(x|\tau^{m-1}t) h_{m-r}(x|t) = 0.\end{equation} Move $h_m(x|t)$ to the left-hand side to write
\begin{equation} h_{m}(x|t)= \sum_{r=1}^{k}(-1)^{r+1}  e_{i}(x|\tau^{i-1}t)h_{r-i}(x|\tau^{1-i}t).\end{equation} Use the inductive hypothesis on $r$ to write each polynomial $h_{r-i}(x|\tau^{1-i}t)$ on the right-hand side in terms of $q$. The above calculations were all written out in the ring of symmetric functions; in our context this implies 
\begin{equation}
\widetilde{\varphi} (h_{nd+j}(x|t)) = (-1)^{d(k-1)}q^d h_j(x|t).\end{equation}
\end{proof}

\begin{proof}[Proof of Proposition \ref{T:AlgProd}]
To establish Proposition \ref{T:AlgProd}, we use the Jacobi-Trudi formula to expand the products of the Schur polynomials $s_{\overline{\lambda}}(x|t) \widehat{s_{\mu}}(x|t)$ and $\widehat{s}_{\lambda^-}(x|t) \widehat{s_{\mu}}(x|t)$.  Then
\begin{align}
M_{\overline{\lambda}} & = \left[ \begin{array}{cccc} h_{n-k+1}(x|t) & h_{n-k+2}(x|\tau^{-1}t) & \cdots& h_n(x | \tau^{1-k}t) \\ h_{\lambda_2-1}(x|t) &\ddots&& h_{\lambda_2+k-1}(x|\tau^{1-k}t) \\ \vdots&& \ddots& \vdots \\ h_{\lambda_k+1-k}(x |  t)& \cdots &\cdots & h_{\lambda_k}(x|\tau^{1-k}t) \end{array} \right], \\
M_{\mu} & =  \left[ \begin{array}{cccc} h_{\mu_1}(x|t) & h_{\mu_1+1}(x|\tau^{-1}t) & \cdots& h_{\mu_1+k-1}(x | \tau^{1-k}t) \\ h_{\mu_2-1}(x|t) & \ddots &  & h_{\mu_2+k-1}(x|\tau^{1-k}t) \\ \vdots&& \ddots& \vdots \\ h_{\mu_k+1-k}(x|t) & \cdots &\cdots & h_{\mu_k}(x|\tau^{1-k}t)\end{array}\right],
\end{align}
and so we can compute that 
\begin{align}\label{E:jacobitrudi}
s_{\overline{\lambda}}(x|t) \widehat{s_{\mu}}(x|t) &= \det \left( M_{\overline{\lambda}} M_{\mu} \right) \\ 
&=\det \left( \sum_{\ell=1}^k h_{\bar{\lambda}_i+i-\ell}(x|\tau^{1-\ell})h_{\mu_{\ell}+j-\ell}(x|\tau^{1-j}t)\right)_{1 \leq i,j \leq k}.
\end{align}
Using Lemma \ref{lem: quantuminvariance} above, notice that $h_{n-k+1+i-\ell}(x|\tau^st)$ is zero modulo the quantum ideal $J$ when $n-k+1+i-\ell$ takes values between $n-k+1$ and $n-1$. Since $\overline{\lambda}_1 = n-k+1$, the only entry in the first row of the matrix $ M_{\overline{\lambda}} M_{\mu}$ that is non-zero modulo $J$ is the last entry.  Expanding the determinant along the first row of the matrix then gives
\begin{multline}\label{eq: expansion} (-1)^k h_n(x|\tau^{1-k}t) \det \left( \sum_{\ell=1}^{k-1} h_{\lambda_{i+1}-1+\ell-i+1}(x| \tau^{1-\ell}t)h_{\mu_{\ell}-\ell+j}(x|\tau^{1-j}t) \right) \\
= 
(-1)^k q \det \left( \sum_{\ell=1}^{k-1} h_{\lambda_{i+1}-1+\ell-i+1}(x| \tau^{1-\ell}t)h_{\mu_{\ell}-\ell+j}(x|\tau^{1-j}t) \right)
\end{multline}
where the equality is given by the second part of Lemma \ref{lem: quantuminvariance}.

By contrast, expanding $\widehat{s}_{\lambda^-}(x|t) \widehat{s_{\mu}}(x|t)$ in a similar fashion gives
\begin{align}
\widehat{s}_{\lambda^-}(x|t) \widehat{s_{\mu}}(x|t) &= \det \left( M_{\lambda^-} M_{\mu} \right)  \\ \nonumber 
&=  \det \left( \sum_{\ell=1}^k h_{\lambda_{i+1}-1+\ell-i+1}(x| \tau^{1-\ell}t)h_{\mu_{\ell}-\ell+j}(x|\tau^{1-j}t) \right).
\end{align}
Since $\lambda_{k+1}=0$, the last row of the matrix $M_{\lambda^-} M_{\mu}$ consists of $k-1$ zeroes followed by a one.  Expanding the determinant then gives 
\begin{equation}
\widehat{s}_{\lambda^-}(x|t) \widehat{s_{\mu}}(x|t) = (-1)^k\det \left( \sum_{\ell=1}^{k-1} h_{\lambda_{i+1}-1+\ell-i+1}(x| \tau^{1-\ell}t)h_{\mu_{\ell}-\ell+j}(x|\tau^{1-j}t) \right) .
\end{equation}
The reduction $\widetilde{\varphi}$ applied to equation \eqref{eq: expansion} is the identity, since the indices of $t_i$ range only between 1 and $n$ (see the definition of factorial Schur polynomial for $\overline{\lambda}$). Thus  
\begin{align} 
\phantom{a} & \widetilde{\varphi}(\widehat{s}_{\overline{\lambda}}(x|t) \widehat{s_{\mu}}(x|t)) \\
\nonumber
& =
(-1)^k q \det \left( \sum_{\ell=1}^{k-1} h_{\lambda_{i+1}-1+\ell-i+1}(x| \tau^{1-\ell}t)h_{\mu_{\ell}-\ell+j}(x|\tau^{1-j}t) \right). 
\end{align}
We see that  $ \widehat{s}_{\lambda^-}(x|t) \widehat{s}_{\mu}(x|t) = \widetilde{\varphi}( \widehat{s}_{\lambda^-}(x|t)\widehat{s_\mu}(x|t))$ and $\widetilde{\varphi} ( s_{\overline{\lambda}}(x|t) \widehat{s_{\mu}}(x|t))$ differ by exactly $q$. Thus, \[\widetilde{\varphi}( s_{\overline{\lambda}}(x|t)\widehat{s_\mu}(x|t))= q \widetilde{\varphi}(\widehat{s}_{\lambda^-}(x|t)\widehat{s_\mu}(x|t)),\]
as desired.
\end{proof}




\bibliographystyle{alphanum}

\bibliography{EqRimhookRefs}

\newcommand{\etalchar}[1]{$^{#1}$}
\begin{thebibliography}{BKPT}

\bibitem[BCFF]{BCFF}
Aaron Bertram, Ionu{\c{t}} Ciocan-Fontanine, and William Fulton.
\newblock Quantum multiplication of {S}chur polynomials.
\newblock {\em J. Algebra}, 219(2):728--746, 1999.

\bibitem[Ber]{Bertram}
Aaron Bertram.
\newblock Quantum {S}chubert calculus.
\newblock {\em Adv. Math.}, 128(2):289--305, 1997.

\bibitem[BKPT]{BKPT}
Anders~Skovsted Buch, Andrew Kresch, Kevin Purbhoo, and Harry Tamvakis.
\newblock The puzzle conjecture for the cohomology of two-step flag manifolds.
\newblock {\em J. Algebraic Combin.}, 44(4):973--1007, 2016.

\bibitem[BKT]{BKT}
Anders~Skovsted Buch, Andrew Kresch, and Harry Tamvakis.
\newblock Gromov-{W}itten invariants on {G}rassmannians.
\newblock {\em J. Amer. Math. Soc.}, 16(4):901--915 (electronic), 2003.

\bibitem[BM]{BuchMihalcea}
Anders~S. Buch and Leonardo~C. Mihalcea.
\newblock Quantum {$K$}-theory of {G}rassmannians.
\newblock {\em Duke Math. J.}, 156(3):501--538, 2011.

\bibitem[Buc]{Buch}
Anders~Skovsted Buch.
\newblock Mutuations of puzzles and equivariant cohomology of two-step flag
  varieties.
\newblock {\em Ann. of Math. (2)}, 182(1):173--220, 2015.

\bibitem[CL]{ChenLouck}
William Y.~C. Chen and James~D. Louck.
\newblock The factorial {S}chur function.
\newblock {\em J. Math. Phys.}, 34(9):4144--4160, 1993.

\bibitem[FP]{FultonPandharipande}
W.~Fulton and R.~Pandharipande.
\newblock Notes on stable maps and quantum cohomology.
\newblock In {\em Algebraic geometry---{S}anta {C}ruz 1995}, volume~62 of {\em
  Proc. Sympos. Pure Math.}, pages 45--96. Amer. Math. Soc., Providence, RI,
  1997.

\bibitem[Ful]{Fulton}
William Fulton.
\newblock {\em Young tableaux}, volume~35 of {\em London Mathematical Society
  Student Texts}.
\newblock Cambridge University Press, Cambridge, 1997.
\newblock With applications to representation theory and geometry.

\bibitem[Giv]{Givental}
Alexander~B. Givental.
\newblock Equivariant {G}romov-{W}itten invariants.
\newblock {\em Internat. Math. Res. Notices}, (13):613--663, 1996.

\bibitem[GK1]{GiventalKim}
Alexander Givental and Bumsig Kim.
\newblock Quantum cohomology of flag manifolds and {T}oda lattices.
\newblock {\em Comm. Math. Phys.}, 168(3):609--641, 1995.

\bibitem[GK2]{GK}
Vassily Gorbounov and Christian Korff.
\newblock Equivariant quantum cohomology and {Y}ang-{B}axter algebras.
\newblock arXiv:1402.2907, 2014.

\bibitem[Gra]{Graham}
William Graham.
\newblock Positivity in equivariant {S}chubert calculus.
\newblock {\em Duke Math. J.}, 109(3):599--614, 2001.

\bibitem[JK]{JamesKerber}
Gordon James and Adalbert Kerber.
\newblock {\em The representation theory of the symmetric group}, volume~16 of
  {\em Encyclopedia of Mathematics and its Applications}.
\newblock Addison-Wesley Publishing Co., Reading, Mass., 1981.
\newblock With a foreword by P. M. Cohn, With an introduction by Gilbert de B.
  Robinson.

\bibitem[KT]{KnutsonTao}
Allen Knutson and Terence Tao.
\newblock Puzzles and (equivariant) cohomology of {G}rassmannians.
\newblock {\em Duke Math. J.}, 119(2):221--260, 2003.

\bibitem[Lak]{LaksovEqPieri}
Dan Laksov.
\newblock A formalism for equivariant {S}chubert calculus.
\newblock {\em Algebra Number Theory}, 3(6):711--727, 2009.

\bibitem[LR]{LiRavikumar}
Changzheng Li and Vijay Ravikumar.
\newblock Equivariant {P}ieri rules for isotropic {G}rassmannians.
\newblock {\em Math. Ann.}, 365(1-2):881--909, 2016.

\bibitem[LS1]{LamShimozono}
Thomas Lam and Mark Shimozono.
\newblock Quantum cohomology of {$G/P$} and homology of affine {G}rassmannian.
\newblock {\em Acta Math.}, 204(1):49--90, 2010.

\bibitem[LS2]{LSDoubleDouble}
Thomas Lam and Mark Shimozono.
\newblock From double quantum {S}chubert polynomials to $k$-double {S}chur
  functions via the {T}oda lattice.
\newblock arXiv:1109.2193, 2011.

\bibitem[LS3]{LSkDouble}
Thomas Lam and Mark Shimozono.
\newblock {$k$}-double {S}chur functions and equivariant (co)homology of the
  affine {G}rassmannian.
\newblock {\em Math. Ann.}, 356(4):1379--1404, 2013.

\bibitem[Mac]{Macdonald}
I.~G. Macdonald.
\newblock {\em Symmetric functions and {H}all polynomials}.
\newblock Oxford Mathematical Monographs. The Clarendon Press Oxford University
  Press, New York, second edition, 1995.
\newblock With contributions by A. Zelevinsky, Oxford Science Publications.

\bibitem[Mih1]{MihalceaAdvances}
Leonardo Mihalcea.
\newblock Equivariant quantum {S}chubert calculus.
\newblock {\em Adv. Math.}, 203(1):1--33, 2006.

\bibitem[Mih2]{MihalceaTransactions}
Leonardo~Constantin Mihalcea.
\newblock Giambelli formulae for the equivariant quantum cohomology of the
  {G}rassmannian.
\newblock {\em Trans. Amer. Math. Soc.}, 360(5):2285--2301, 2008.

\bibitem[MS]{MolevSagan}
Alexander~I. Molev and Bruce~E. Sagan.
\newblock A {L}ittlewood-{R}ichardson rule for factorial {S}chur functions.
\newblock {\em Trans. Amer. Math. Soc.}, 351(11):4429--4443, 1999.

\bibitem[Pos]{Postnikov}
Alexander Postnikov.
\newblock Affine approach to quantum {S}chubert calculus.
\newblock {\em Duke Math. J.}, 128(3):473--509, 2005.

\bibitem[Rie]{Rietsch}
Konstanze Rietsch.
\newblock Quantum cohomology rings of {G}rassmannians and total positivity.
\newblock {\em Duke Math. J.}, 110(3):523--553, 2001.

\bibitem[S{\etalchar{+}}]{sage}
W.\thinspace{}A. Stein et~al.
\newblock {\em {S}age {M}athematics {S}oftware ({V}ersion 5.10)}.
\newblock The Sage Development Team, 2013.
\newblock {\tt http://www.sagemath.org}.

\bibitem[San]{Santiago}
Ta\'{i}se Santiago.
\newblock {S}chubert calculus on a {G}rassmann algebra.
\newblock Ph.D. thesis, {P}olitecnico di {T}orino, 2006.

\end{thebibliography}

\end{document}